\documentclass[twoside,11pt]{amsart}
\usepackage[all]{xy}
\CompileMatrices
\usepackage{amsfonts}
\usepackage{amssymb}
\usepackage{amsthm}
\usepackage{amsmath}
\usepackage{ifthen}
\usepackage{enumitem} 
\usepackage{verbatim} 
\usepackage{url}
\usepackage[usenames]{color}

\allowdisplaybreaks 

 \newtheorem{thm}{Theorem}[section]
 \newtheorem{prop}[thm]{Proposition}
 \newtheorem{cor}[thm]{Corollary}
 \newtheorem{lem}[thm]{Lemma}

\theoremstyle{definition}
\newtheorem{defn}[thm]{Definition}

\theoremstyle{remark}
\newtheorem{rem}[thm]{Remark}

\newcommand{\Z}{\mathbb{Z}}
\newcommand{\Q}{\mathbb{Q}}
\newcommand{\R}{\mathbb{R}}
\newcommand{\C}{\mathbb{C}}
\renewcommand{\Im}{\mathrm{Im}}
\newcommand{\A}{\mathbb{A}}
\newcommand{\HH}{\mathbb{H}}
\renewcommand{\H}{\mathcal{H}}
\newcommand{\Sp}{\mathrm{Sp}}
\newcommand{\GL}{\mathrm{GL}}
\newcommand{\Jac}{\b{G}}
\renewcommand{\Re}{\mathrm{Re}}
\renewcommand{\Im}{\mathrm{Im}}

\newcommand{\diag}{\mathrm{diag}}
\newcommand{\sgn}{\mathrm{sgn}}

\newcommand{\tr}{\mathrm{tr}\,}
\renewcommand{\b}[1]{\boldsymbol{#1}} 
\renewcommand{\a}{\mathbf{a}} 
\newcommand{\h}{\mathbf{h}} 
\newcommand{\f}[1]{\mathfrak{{#1}}}

\renewcommand{\l}{\lambda}
\renewcommand{\t}{\theta}

\newcommand{\lag}{\l_\Gamma}
\newcommand{\lagp}{\l_{\Gamma'}}

\newcommand{\back}{\backslash}
\newcommand{\T}[1]{\,{}^t\! {{#1}}} 
\newcommand{\transpose}[1]{\text{$^t\!#1$}} 
\renewcommand{\(}{\left(} \renewcommand{\)}{\right)}

\def\sectionnam{\@empty}
\def\subsectionnam{\@empty}

\begin{document}
\title[Algebraicity of special $L$-values]{Algebraicity of special $L$-values attached to Siegel-Jacobi modular forms}


\author{Thanasis Bouganis and Jolanta Marzec}
\address{Department of Mathematical Sciences\\ Durham University\\
 Durham, UK. \newline
 Institute of Mathematics \\
 TU Darmstadt \\
 Darmstadt, Germany}
\email{athanasios.bouganis@durham.ac.uk\\ marzec@mathematik.tu-darmstadt.de}
\thanks{The authors acknowledge support from EPSRC through the grant EP/N009266/1, Arithmetic of automorphic forms and special $L$-values}

\maketitle In this work we obtain algebraicity results on special $L$-values attached to Siegel-Jacobi modular forms.  
Our method relies on a generalization of the doubling method to the Jacobi group obtained in our previous work, and on introducing a notion of near holomorphy for Siegel-Jacobi modular forms. Some of our results involve also holomorphic projection, which we obtain by using Siegel-Jacobi Poincar\'e series of exponential type.



\section{Introduction} 
\definecolor{paleyellow}{RGB}{255,255,190}
\newcommand{\yel}[1]{\colorbox{paleyellow}{#1}}

This paper should be seen as a continuation of our earlier paper \cite{BMana} on properties of the standard $L$-function attached to a Siegel-Jacobi modular form. Indeed, in \cite{BMana} we have established various analytic properties (Euler product decomposition, analytic continuation and detection of poles) of the standard $L$-function attached to Siegel-Jacobi modular forms, and in this paper we turn our attention to algebraicity properties of some special $L$-values. 

Shintani was the first person who attached
an $L$-function to a Siegel-Jacobi modular form which is an eigenfunction of a properly defined Hecke algebra. He initiated the study of its analytic properties by finding an integral representation. His work was left unpublished, but then was took over by Murase \cite{Mu89, Mu91} and Arakawa \cite{Ar94} who obtained results on the analytic properties of this $L$-function using variants of the doubling method. In our previous work \cite{BMana} we extended their results to a very general setting: non-trivial level, character and a totally real algebraic number field. For this purpose we applied the doubling method to the Jacobi group, and consequently related Siegel-type Jacobi Eisenstein series to the standard $L$-function. This identity has a further application in the current paper.

Here the starting point of our investigation 
is a result of Shimura in \cite{Sh78} on the arithmeticity of Siegel-Jacobi modular forms. Namely, if we let $S$ be a positive definite half-integral $l$ by $l$ symmetric matrix, write $M_{k,S}^n$ for the space of Siegel-Jacobi modular forms of weight $k$ and index $S$ (see next section for a definition), and of any congruence subgroup, and denote by $M_{k,S}^n(K)$ the subspace of  $M_{k,S}^n$ consisting of those functions whose Fourier expansion at infinity has Fourier coefficients
in a subfield $K$ of $\mathbb{C}$, then it is shown in  (loc. cit.) that  $M_{k,S}^n(K)= M_{k,S}^n(\mathbb{Q}) \otimes_{\mathbb{Q}} K$. In particular, for a given $f \in M_{k,S}^n$ and a $\sigma \in Aut(\mathbb{C}/\mathbb{Q})$ one can define the element $f^{\sigma} \in M_{k,S}^n$ by letting $\sigma$ act on the Fourier coefficients of $f$. 

The main result of this paper comes in two flavours: Theorem \ref{Main Theorem on algebraicity} and Theorem \ref{Main Theorem on algebraicity v2}. Without going into too much details, it may be vaguely stated as follows. Denote by $L(s,f,\chi)$ the standard $L$-function attached to a Siegel-Jacobi cuspidal eigenform $f \in M_{k,S}^n(\overline{\mathbb{Q}})$, which is twisted by a Dirichlet character $\chi$, and let 
\[
\b{\Lambda}(s,f,\chi) := L(2s-n-l/2,f,\chi) \begin{cases} L_{\mathfrak{c}}(2s-l/2,\chi \psi_S) & \hbox{if } l \in 2\mathbb{Z},\\
1 & \hbox{if } l \notin 2\mathbb{Z},
 \end{cases}
\]
where $\psi_S$ is the non-trivial quadratic character attached to the extension $K_S:= \mathbb{Q}(\sqrt{(-1)^{l/2} \det(2S)})$ if $K_S \neq \mathbb{Q}$, and otherwise $\psi_S =1$. 
Then for certain integers $\sigma$ and for $k \gg 0$,
\[
\frac{\b{\Lambda}(\sigma/2,f,\chi)}{\pi^{e_{\sigma}} <f,f>} \in \overline{\mathbb{Q}},
\]
for an explicit power $e_{\sigma} \in \mathbb{N}$, and where $<f,f>$ is a Petersson inner product on the space of cuspidal Siegel-Jacobi modular forms.

In Theorem \ref{Main Theorem on algebraicity}, where the cusp form $f$ is defined over a totally real field $F$, we can take above any $k>2n+l+1$, but we have to assume that in the theta decomposition of $f$ the modular forms of (half-)integral weight are all cuspidal, which we call Property A in this paper. It is very tempting to say that this last assumption always holds (see examples and remark on pages \pageref{ex:property A}-\pageref{rem:property A}), but nevertheless it may not be so easy to check in practice. Theorem \ref{Main Theorem on algebraicity v2} seeks to improve this situation, and does not rely on Property A. This second theorem is obtained under a slightly weaker bound on the weight, namely $k>6n+2l+1$, and for simplicity we prove it only for $F=\mathbb{Q}$. 

We would like to emphasize that these two results come with very different methods, although they both base on the doubling identity. The key difficulty which forces the aforementioned assumptions is the existence of a projection from nearly holomorphic to cuspidal Siegel-Jacobi forms. In the first case this is resolved by employing the known projection from (nearly) holomorphic Siegel modular forms (see Sections \ref{sec:arithmetic} and \ref{sec:nhol}). In the second case, modelled on the work \cite{St} of Sturm, we study Siegel-Jacobi Poincar\'e series of exponential type and define the projection via the associated reproducing kernel (Section \ref{sec:Sturm}); this might be of independent interest. The latter approach involves also a twist in the proof, which uses the theory of CM points in the Jacobi setting.

Let us now try to put the main result of this paper in some broader context. Results  of the above form for the standard $L$-functions of automorphic forms associated to Shimura varieties, such as Siegel and Hermitian modular forms, were obtained by many researchers, most profoundly by Shimura (see for example \cite{Sh00}). These results can also be understood in the general framework of Deligne's Period Conjectures for critical values of motives \cite{Deligne}. Indeed, according to the general Langlands conjectures, the standard $L$-functions of automorphic forms related to Shimura varieties can be identified with motivic $L$-functions, and hence the algebraicity results for the special values of the automorphic $L$-functions can also be seen as a confirmation of Deligne's Period Conjecture, albeit is usually hard to actually show that the conjectural motivic period agrees with the automorphic one.

Siegel-Jacobi modular forms and - in particular - the algebraicity results obtained in this paper do not fit in this framework: since the Jacobi group is not reductive, it does not satisfy the necessary properties to be associated with a Shimura variety. Nevertheless, the Jacobi group can be actually associated with a geometric object, namely with a mixed Shimura variety, as it is explained for example in \cite{Kramer,Milne}. It is thus very tempting to speculate - although we do not address this question in this paper - that the standard $L$-function studied here might be identified with an $L$-function of a mixed motive, and hence the theorem above could be seen as a confirmation of the generalization of Deligne's Period Conjecture to the mixed setting as for example stated by Scholl in \cite{Scholl}. 

We should mention here that even though in some cases one can identify the standard $L$-function associated to a Siegel-Jacobi form with the standard $L$-function associated to a Siegel modular form (see for example the remark on page 252 in \cite{Mu91}), this is possible under some quite restrictive conditions on both index and level of the Siegel-Jacobi form. Actually, even in the situation of classical Jacobi forms this correspondence becomes quite complicated when one considers an index different than $1$ and/or non-trivial level, which is very clear for example in the work of \cite{SZ}. 

\noindent \textbf{Remark:} In an earlier version of \cite{BMana}, which one can find on the arXiv (\cite{BMarXiv}), we had also included the results of this paper (except the last two sections). However this had resulted in a rather long exposition, and for this reason we decided to keep the two main results of our investigations separately. Namely, \cite{BMana} contains now our results towards the analytic properties of the standard $L$ function, whereas this paper focuses on the algebraic properties.


\section{Preliminaries}
In this section we recall basic facts regarding Siegel-Jacobi modular forms of higher index and set up the notation. We follow closely our previous work \cite{BMana}.

Let $F$ be a totally real algebraic number field of degree $d$, $\f{d}$ the different of $F$, and $\f{o}$ its ring of integers. For two natural numbers $l,n$, we consider the Jacobi group $\b{G}:=\b{G}^{n,l}:=H^{n,l}\rtimes \Sp_n$ of degree $n$ and index $l$ over $F$:
$$\b{G}^{n,l}(F):=\{\b{g}=(\l,\mu,\kappa)g: \l,\mu\in M_{l,n}(F), \kappa\in Sym_l(F), g\in G^n(F)\} ,$$
where $H(F):=H^{n,l}(F):=\{ (\l,\mu,\kappa) 1_{2n}\in\b{G}^{n,l}(F)\}$ is the Heisenberg group, and  
$$G^n(F):=\Sp_n(F):=\left\{ g\in \mathrm{SL}_{2n}(F)\colon\T{g}\(\begin{smallmatrix}  &  -1_n   \\  1_n  & \\ \end{smallmatrix}\) g=\(\begin{smallmatrix}  &  -1_n   \\  1_n  &    \\ \end{smallmatrix}\)\right\} .$$
The group law is given by
$$(\l,\mu,\kappa)g (\l',\mu',\kappa')g':=(\l+\tilde{\l},\mu+\tilde{\mu}, \kappa+\kappa'+\l\T{\tilde{\mu}}+\tilde{\mu}\T{\l}+\tilde{\l}\T{\tilde{\mu}}-\l'\T{\mu'}) gg',$$
where $(\tilde{\l}\,\tilde{\mu}):=(\l'\,\mu')g^{-1}=(\l'\T{d}-\mu'\T{c}\quad \mu'\T{a}-\l'\T{b})$,
and the identity element of $\b{G}^{n,l}(F)$ is $1_H1_{2n}$, where $1_H:=(0,0,0)$ is the identity element of $H^{n,l}(F)$ (whenever it does not lead to any confusion we suppress the indices $n,l$).
For an element $g \in \Sp_n$ we write $g = \left( \begin{smallmatrix} a_g & b_g \\ c_g & d_g \end{smallmatrix}\right)$, where $a_g,b_g,c_g,d_g \in M_n$.

We write  $\{\sigma_v:F\hookrightarrow\R,\,\, v \in \mathbf{a}\}$ for the set of real embeddings of $F$ and put $F_\a:=\prod_{v\in\a} \sigma_v (F)$, $\a$ denoting the set of archimedean places of $F$. Each $\sigma_v$ induces an embedding $\b{G}(F) \hookrightarrow \b{G}(\mathbb{R})$; we will write $(\l_v,\mu_v,\kappa_v)g_v$ for $\sigma_v(\b{g})$. 
The group $\Jac(\mathbb{R})^{\mathbf{a}}:= \prod_{v \in \mathbf{a}} \b{G}(\mathbb{R})$ acts on $\H_{n,l}:=(\mathbb{H}_n\times M_{l,n}(\C))^{\a}$ component wise via
\[
\b{g} z  = \b{g}(\tau,w)=(\l,\mu,\kappa)g (\tau,w)=\prod_{v\in\a} (g_v\tau_v, w_v \lambda(g_v,\tau_v)^{-1}+\l_v g_v \tau_v+\mu_v),
\]
where $g_v\tau_v=(a_v\tau_v+b_v)(c_v\tau_v+d_v)^{-1}$ and $\lambda(g_v,\tau_v) :=(c_v\tau_v+d_v)$ for $g_v=\(\begin{smallmatrix} a_v & b_v \\ c_v & d_v \end{smallmatrix}\)$.

For $k = (k_v) \in \Z^{\a} := \prod_{v \in \mathbf{a}} \mathbb{Z}$ and a matrix $S\in Sym_l(\f{d}^{-1})$ we define the factor of automorphy of weight $k$ and index $S$ by
$$J_{k,S}\colon\Jac^{n,l}(F)\times\H_{n,l}\to \C$$
$$J_{k,S} (\b{g},z)=J_{k,S} (\b{g},(\tau,w)):=\prod_{v\in\a} j(g_v,\tau_v)^{k_v} \mathcal{J}_{S_v}(\b{g_v},\tau_v,w_v) ,$$
where $\b{g}=(\l,\mu,\kappa)g$, $j(g_v,\tau_v)=\det(c_v\tau_v+d_v)=\det(\lambda (g_v,\tau_v))$ and
\begin{multline*}
\mathcal{J}_{S_v}(\b{g}_v,\tau_v,w_v)=e(-\tr (S_v\kappa_v)+\tr (S_v[w_v] \lambda(g_v,\tau_v)^{-1}c_v)\\
-2\tr (\T{\l}_v S_v w_v \lambda(g_v,\tau_v)^{-1})-\tr (S_v[\l_v] g_v \tau_v))
\end{multline*}
with $e(x):=e^{2\pi ix}$, and we set $S[x]:=\T{x}Sx$; $J_{k,S}$ satisfies the usual cocycle relation: 
$$
J_{k,S} (\b{g}\b{g}',z)=J_{k,S} (\b{g},\b{g}'\,z)J_{k,S} (\b{g}',z).
$$
For a function $f \colon\H_{n,l}\to\C$ we define the action of $\b{g}\in\b{G}^{n,l}$ by
$$(f|_{k,S}\, \b{g})(z):=J_{k,S} (\b{g},z)^{-1}f(\b{g}\,z).$$

One can define a very general notion of a congruence subgroup of $\b{G}(F)$ by considering any congruence subgroup $\Gamma$ of $Sp_n(F)$ and a lattice in $H^{n,l}(F)$ that is stable under the action of $\Gamma$. However
in this paper we will mainly focus on functions $f$ which are invariant under the action of a particular congruence subgroup of $\b{G}(F)$, namely, 
$$\b{\Gamma}[\mathfrak{b},\mathfrak{c}]:=C[\mathfrak{o},\mathfrak{b}^{-1},\mathfrak{b}^{-1}] \rtimes \Gamma[\mathfrak{b}^{-1},\mathfrak{bc}]\subset\b{G}(F),$$
where $\f{b}$ and $\f{c}$ are respectively a fractional and an integral ideal of $F$, and
$$C[\mathfrak{o},\mathfrak{b}^{-1},\mathfrak{b}^{-1}] := \{ (\lambda,\mu,\kappa) \in H^{n,l}(F):\lambda\in M_{l,n}(\mathfrak{o}), \mu\in M_{l,n}(\mathfrak{b}^{-1}), \kappa\in Sym_l( \mathfrak{b}^{-1})\} ,$$
$$\Gamma[\mathfrak{b}^{-1},\mathfrak{bc}]:=\left\{ g = \left(\begin{matrix} a_g & b_g  \\ c_g & d_g \end{matrix} \right) \in G(F):\begin{smallmatrix} a_g\in M_n(\f{o}), & b_g\in M_n(\mathfrak{b}^{-1}),\\ c_g\in M_n(\mathfrak{b}\mathfrak{c}), & d_g\in M_n(\f{o})\end{smallmatrix}\right\} .$$

Now, let $S \in \mathfrak{b}\mathfrak{d}^{-1}\mathcal{T}_l$ where 
$$\mathcal{T}_l:= \{ x \in Sym_l(F):\tr(xy)\in\mathfrak{o} \mbox{ for all } y \in Sym_l(\mathfrak{o}) \} ,$$
and assume additionally that $S$ is positive definite in the sense that if we write $S_v := \sigma_v(S) \in Sym_{l}(\mathbb{R})$ for $v \in \mathbf{a}$, then all $S_v$ are positive definite. 

\begin{defn}\label{definition modular Jacobi}
Let $k$ and $S$ be as above, and $\b{\Gamma}$ a subgroup of $\b{\Gamma}[\mathfrak{b},\mathfrak{c}]$ (for some ideals $\f{b},\f{c}$) equipped with a homorphism $\chi$. 
A Siegel-Jacobi modular form of weight $k\in\Z^{\a}$, index $S$, level $\b{\Gamma}$ and Nebentypus $\chi$ is a holomorphic function $f\colon\H_{n,l}\to\C$ such that
	\begin{enumerate}
	\item[i)] \label{def:Phi} $f|_{k,S}\, \b{g} =\chi(\b{g})f$ for every $\b{g} \in \b{\Gamma}$;
	\item[ii)] for each $g\in G^n(F)$, $f|_{k,S}\, g$ admits a Fourier expansion 
of the form
\[
f|_{k,S}\, g(\tau,w)=\sum_{\substack{t\in L\\ t\geq 0}}\sum_{r\in M} c(g;t,r) \mathbf{e}_{\mathbf{a}}(\tr(t\tau)) \mathbf{e}_{\mathbf{a}}(\tr(\T{r}w))\,\,\,\,(*)
\]
for some appropriate lattices $L \subset Sym_n(F)$ and $M \subset M_{l,n}(F)$, and such that $\begin{pmatrix} S & r \\ \transpose{r} & t \end{pmatrix}\geq 0$, i.e., this is a semi-positive definite matrix for each $v\in\a$; we set $\mathbf{e}_{\mathbf{a}}(x):=\prod_{v\in\a} e(x_v)$ for $x=\prod_{v\in\a} x_v$.\\ 
\end{enumerate}
We will denote the space of such functions by $M^n_{k,S}(\b{\Gamma} ,\chi)$.	
\end{defn}

\begin{rem} Thanks to the K\"{o}cher principle for Siegel-Jacobi forms the second condition is automatic if $n \geq 2$ or $F\neq \mathbb{Q}$ and hence it needs to be imposed only if $n=1$ and $F = \mathbb{Q}$. Indeed, this is shown in \cite[Lemma 1.6]{Z89} or \cite[Proposition 3.1]{Sh78} for the case of $n \geq 2$ and $F=\mathbb{Q}$ but it is very clear that its proof extends to the case of $F \neq \mathbb{Q}$ since it only relies to the K\"ocher principle for Siegel modular forms which is true also for totally real fields (see for example \cite[page 31]{Sh00}. For the case of $n=1$ and $F$ a totally real field different to $\mathbb{Q}$ one can use the same proof but now use the K\"ocher principle for Hilbert modular forms (see also \cite[Lemma 3.50]{Boylan}).
\end{rem}

We say that $f$ is a cusp form if in the expansion $(*)$ above for every $g \in G^n(F)$, we have $c(g;t,r) = 0$ unless $\begin{pmatrix} S_v & r_v \\ \transpose{r_v} & t_v \end{pmatrix}$ is positive definite for every $v \in \mathbf{a}$. The space of cusp forms will be denoted by $S^n_{k,S}(\b{\Gamma},\chi)$.

We define Petersson inner product of Siegel-Jacobi forms $f$ and $g$ of weight $k$ and level $\b{\Gamma}$ under assumption that one of them is a cusp form as:
\[
<f,g>:= vol(A)^{-1} \int_{A} f(z) \overline{g(z)} \Delta_{S,k}(z) dz, \,\,\,\, A:= \b{\Gamma} \setminus \mathcal{H}_{n,l} ,
\]
where for $z = (\tau,w) \in \mathcal{H}_{n,l}$, $\tau = x + i y$ with $x,y \in Sym_n(F_{\a})$ and $w = u + i v$ with $u,v \in M_{l,n}(F_{\a})$, we set
$$dz := d(\tau,w):= \det(y)^{-(l+n+1)}dx dy du dv\, ,\quad\Delta_{S,k}(z):= \det(y)^k \exp_\a(-4 \pi \tr(S[v] y^{-1})).$$ 
We also recall (see \cite[p. 187]{Ar94}) that $\Delta_{S,k}(\b{g}z)=|J_{k,S}(\b{g},z)|^{-2}\Delta_{S,k}(z)$ for any $\b{g}\in\b{G}^{n,l}(F)$.
In this way the inner product is independent of the group $\b{\Gamma}$. 

We finish this section with a final remark. In this paper we consider the $L$ function (see next section) attached to integral weight Siegel-Jacobi form i.e. $k = (k_v)$ with $k_v \in \mathbb{Z}$. However some half-integral weight Siegel-Jacobi forms will arise quite naturally in the paper, and in particular theta series. That is we will have to consider weights of the form $k = (k_v)$ with all $k_v \in \frac{1}{2}\mathbb{Z} \setminus \mathbb{Z}$. In this case the definition is similar to the one given above (i.e. Definition \ref{definition modular Jacobi}, but one replaces $j(g_v,\tau_v)^{k_v}$ in the definition of $J_{k,S}(\b{g},z)$ above, with the corresponding half-integral one as for example is defined in \cite[paragraph 6.10]{Sh00}.

\section{The standard $L$-function and the doubling method identity}\label{sec:L-function}
In this section we recall some results and notation from \cite{BMana} which will be necessary to establish results in the next section.

\subsection{The $L$-function}
We start by fixing some notation. For a fractional ideal $\f{b}$, and an integral ideal $\f{c}$ we let
$$\b{\Gamma} := \b{\Gamma}_1(\mathfrak{c}):=\{ (\l,\mu,\kappa)g\in C[\f{o}, \f{b}^{-1}, \f{b}^{-1}]\rtimes \Gamma[\f{b}^{-1}\f{c},\f{bc}]:a_g-1_n\in M_n(\f{c})\} ,$$

For any integral ideal $\mathfrak{a}$ of $F$ we have defined in \cite[Section 7]{BMana} a Hecke operator $T(\mathfrak{a}) : S^n_{k,S}(\b{\Gamma})\rightarrow S^n_{k,S}(\b{\Gamma})$.
We now consider a nonzero $f\in S^n_{k,S}(\b{\Gamma})$ such that $f|T(\f{a}) = \lambda(\f{a}) f$ for all integral ideals $\f{a}$ of $F$. For a Hecke character $\chi$ of $F$, and denoting by $\chi^*$ the corresponding ideal character, we define an absolutely convergent series
\[
D(s,f,\chi) := \sum_{\f{a}} \lambda(\f{a}) \chi^{*}(\f{a}) N(\f{a})^{-s},\qquad\Re(s) > 2n+l+1.
\]
In \cite{BMana} we proved a theorem regarding the Euler product representation of this Dirichlet series. The theorem is subject to a condition on the matrix $S$, which first appeared in \cite[page 142]{Mu89}, and may be stated as follows. Consider any prime ideal $\f{p}$ of $F$ such that $(\f{p},\f{c}) =1$ and write $v$ for the corresponding finite place of $F$. We say that the lattice $L := \mathfrak{o}_v^{l} \subset F_{v}^l$ is an $\f{o}_{v}$-maximal lattice with respect to a symmetric matrix $2S$ if for every $\mathfrak{o}_{v}$ lattice $M$ of $F_{v}^l$ that contains $L$ and satisfies $S[x] \in \f{o}_{v}$ for all $x \in M$, we have $M=L$. For any uniformiser $\pi$ of $F_{v}$ we now set 
\[
L' := \{ x \in (2S)^{-1} L: \pi S[x] \in \f{o}_{v}\} \subset F_{v}^l.
\]
We say that the matrix $S$ satisfies the condition $M_{\f{p}}^+$ if $L$ is an $\f{o}_{v}$-maximal lattice with respect to the symmetric matrix $2S$ and $L=L'$.

\begin{thm}[Theorem 7.1, \cite{BMana}]\label{Euler Product Representation} Let $0 \neq f\in S^n_{k,S}(\b{\Gamma})$ be such that $f|T(\f{a}) = \lambda(\f{a}) f$ for all integral ideals $\f{a}$ of $F$.  Assume that the matrix $S$ satisfies the condition $M_{\f{p}}^+$ for every prime ideal $\f{p}$ with $(\f{p},\f{c})=1$. Then
\[
\f{L}(\chi,s) D(s+n+l/2,f,\chi) = L(s,f,\chi) := \prod_{\f{p}} L_{\f{p}}(\chi^{*}(\f{p})N(\f{p})^{-s})^{-1},
\]
where for every prime ideal $\f{p}$ of $F$  
\[
 L_{\f{p}}(X) = \begin{cases} \prod_{i=1}^n \left((1- \mu_{\f{p},i}X)(1-\mu^{-1}_{\f{p},i}X ) \right)\! ,\,\,\, \mu_{\f{p},i} \in \mathbb{C}^{\times} & \hbox{if } (\f{p}, \mathfrak{c})=1,\\
 1 & \hbox{if } (\f{p}, \mathfrak{c})\neq 1.
 \end{cases}
\]
Moreover, $\f{L}(\chi,s) = \prod_{(\f{p},\f{c})=1} \f{L}_{\f{p}}(\chi,s)$, where
\[
\f{L}_{\f{p}}(\chi,s) := G_{\f{p}}(\chi,s) \cdot  \begin{cases}  \prod_{i=1}^{n} L_{\f{p}}(2s+2n-2i,\chi^{2}) &\mbox{if } l \in 2\mathbb{Z}\\ 
\prod_{i=1}^{n} L_{\f{p}}(2s+2n-2i+1,\chi^{2})  & \mbox{if }  l \not \in 2\mathbb{Z}.\end{cases}
\]
Here, for a Hecke Character $\psi$, we let $L_{\f{p}}(s, \psi) := 1-\psi^*(\f{p}) N(\f{p})^{-s}$, and $G_{\f{p}}(\chi,s)$ is a ratio of Euler factors which for almost all $\f{p}$ is equal to one. In particular, the function  $L(s,f,\chi)$ is absolutely convergent for $\Re(s) > n + l/2 +1$.
\end{thm}

We note here that the Euler product expression implies that 
\begin{equation}\label{eq:non-vanishing of L-values}
L(s,f,\chi) \neq 0,\qquad \Re(s) > n + l/2 +1.
\end{equation}

For the rest of the paper, whenever the $L$ function attached to a Jacobi form $f$ is considered, \textbf{we will always assume} that the index matrix $S$ of $f$ satisfies the condition $M_{\f{p}}^+$ for every prime ideal $\f{p}$ away from the level $\mathfrak{c}$ of $f$.

Let us now also remark that if we replace $f$ with $f^c(z) := \overline{f(-\overline{z})}$, the $L$-function remains the same:

\begin{prop}[Proposition 7.9, \cite{BMana}]\label{Behaviour under complex conjugation} 
Let $f\in S^n_{k,S}(\b{\Gamma})$ be an eigenform with $f| T(\f{a}) = \lambda(\f{a}) f$ for all fractional ideals $\f{a}$ prime to $\f{c}$. Then so is $f^c$. In particular, $f^c |T(\f{a}) = \lambda(\f{a}) f^c$ and $L(s,f,\chi) = L(s,f^c,\chi)$. 
\end{prop}

\subsection{Doubling method}
The $L$-function introduced above may be also obtained via a doubling method. We chose to take Arakawa's approach \cite{Ar94} and considered a homomorphism  
$$\iota_A\colon\Jac^{m,l}\times\Jac^{n,l}\to\Jac^{m+n,l},$$ 
$$\iota_A((\l,\mu,\kappa)g\times (\l',\mu',\kappa')g'):= ((\l\, \l'), (\mu\, \mu'),\kappa+\kappa')\iota_S(g\times g'),$$
where 
$$\iota_S:G^m\times G^n\hookrightarrow G^{m+n},\quad \iota_S\(
\(\begin{smallmatrix}
a  & b  \\ c  & d  \\
\end{smallmatrix}\)\times
\(\begin{smallmatrix}
a'  & b'  \\ c'  & d' \\
\end{smallmatrix}\)\) :=
\(\begin{smallmatrix}
a	&	 &  b  &    \\
	& a' &     & b' \\
c	&	 & d   &    \\
	& c' &     & d' \\
\end{smallmatrix}\) .$$
The map $\iota_A$ induces an embedding 
\[
\mathcal{H}_{m,l} \times \mathcal{H}_{n,l} \hookrightarrow \mathcal{H}_{n+m,l},\,\,\,z_1 \times z_2 \mapsto \diag[z_1,z_2],
\]
defined by 
\[
(\tau_1,w_1) \times  (\tau_2,w_2) \mapsto  (\diag[\tau_1,\tau_2], (w_1 \, w_2)).
\] 
The doubling method suggests that computation of the
Petersson inner product of a cuspidal Siegel-Jacobi modular form $f$ on $\mathcal{H}_{n,l}$ against a Siegel-type Jacobi Eisenstein series pull-backed from $\mathcal{H}_{n+m,l}$ leads to an $L$-function associated with $f$. Before we state the result, we define the Jacobi Eisenstein series which was used. It will appear again in the last section (Theorem \ref{algebraic properties of Eisenstein series}), where the question of its nearly holomorphicity will be addressed. In order to avoid excessive and unnecessary notation, we state here an adelic definition of the Eisenstein series; the interested reader is referred to \cite{BMana} for a derivation of a classical definition.

Fix a weight $k \in \Z^{\a}$ and consider a Hecke character $\chi$ such that for a fixed integral ideal $\f{c}$ of $F$ we have
\begin{enumerate}
	\item  $\chi_v(x) = 1$ for all $x \in \f{o}_v^{\times}$ with $x-1 \in \f{c}_v$, $v\in\h$ (finite places of $F$),
	\item $\chi_{\a}(x_{\a}) = \sgn (x_{\a})^k:= \prod_{v \in \mathbf{a}} \left(\frac{x_v}{|x_v|}\right)^{k_v}$, for $x_{\a}  = (x_v) \in \prod_{v \in \mathbf{a}} \mathbb{R}$;
\end{enumerate}
we will also write $\chi_{\f{c}}:=\prod_{v|\f{c}}\chi_v$. Further, let 
$$K^n:=K_{\h}[\f{b},\f{c}] (\prod_{v\in\a} H^{n,l}(F_v) \rtimes D_{\infty}),$$
where
$K_{\h}[\f{b},\f{c}]\subset\prod^{'}_{v\in\h} \b{G}(F_v)$ is defined so that $\b{\Gamma}[\f{b},\f{c}]=\b{G}(F)\cap \prod_{v\in\a} \b{G}(F_v)K_{\h}[\f{b},\f{c}]$ and $D_{\infty}$ is a maximal compact subgroup of $\Sp_n(\mathbb{R})$.

We define an absolutely convergent adelic Eisenstein series of Siegel type on a Jacobi group with a parabolic subgroup
$$\b{P}^n(F):=\left\{ (0,\mu,\kappa )g:\mu\in M_{l, n}(F),\kappa\in Sym_l(F), g\in P^n(F)\right\} ,$$
where $P^n(F)$ is a Siegel subgroup of $G^n(F)$ as follows:
$$E^n(x,s;\chi):= \sum_{\gamma \in\b{P}^n(F) \setminus \b{G}^n(F)} \phi(\gamma x,s;\chi),\quad\Re(s) >\frac12\( n+l+1\) ,$$
where $\phi(x,s;\chi) := 0$ if $ x \notin \b{P}^n(\mathbb{A})K^n$ and otherwise, if $x = \b{p}\b{w}$ with $\b{p} \in \b{P}^n(\mathbb{A})$ and $\b{w} \in K^n$, we set
$$
\phi(x,s;\chi):= \chi(\det(d_p))^{-1} \chi_{\mathfrak{c}}(\det(d_w))^{-1} J_{k,S}(\b{w},\mathbf{i}_0)^{-1} |\det(d_p)|_{\A}^{-2s},  
$$
where $p,w \in \Sp_n(\A)$ denote symplectic parts of $\b{p},\b{w}$, respectively. The classical Jacobi Eisenstein series which corresponds to $E^n(x,s;\chi)$ is given by 
$$E^n(z,s;\chi):=J_{k,S}(x,\b{i}_0)E^n(x,s;\chi),$$
where $\b{i}_0=(i1_n,0)^{\a}\in\mathcal{H}_{n,l}$, and $x\in\prod_{v\in\a}\b{G}(F_v)$ is such that $x\b{i}_0=z$; for a formula which does not involve $x$ see \cite[equation (11)]{BMana}.

\begin{thm}[\cite{BMana}]\label{thm:dm_identity}
Let $f \in S_{k,S}^n(\b{\Gamma})$ be a Hecke eigenform whose index $S$ satisfies the condition $M_{\f{p}}^+$ for all prime ideals $\f{p}$ of $F$ coprime to $\f{c}$, and let $E^{2n}(z,s;\chi)$ be the Eisenstein series defined above. Then, there exists an $\b{\rho} \in \b{G}(F)$ such that:
\begin{align}\label{eq:dm_identity}\nonumber
G&(\chi,2s-n-l/2) N(\f{b})^{2ns}\chi_{\h}(\theta)^{-n}(-1)^{n(s-k/2)} vol(A)\Lambda_{k-l/2,\f{c}}^{2n}(s-l/4,\chi\psi_S)\\
&\cdot <(E^{2n}|_{k,S} \b{\rho}) (\diag[z_1,z_2],s;\chi),(f|_{k,S}\b{\eta}_n)^c(z_2)>\nonumber\\
&\hspace{8cm}=\nu_{\f{e}}c_{S,k}(s-k/2) \b{\Lambda}(s,f,\chi) f(z_1),
\end{align}
where 
$$\Lambda^{2n}_{k-l/2,\mathfrak{c}}(s-l/4,\chi\psi_S) = \begin{cases} L_{\mathfrak{c}}(2s-l/2,\chi \psi_S) \prod_{i=1}^{n} L_{\mathfrak{c}}(4s-l-2i,\chi^{2}) & \hbox{if } l \in 2\mathbb{Z},\\
 \prod_{i=1}^{[(2n+1)/2]} L_{\mathfrak{c}}(4s-l-2i+1,\chi^{2}) & \hbox{if } l \notin 2\mathbb{Z},
 \end{cases}$$
$$\b{\Lambda}(s,f,\chi)\! :=\! L(2s-n-l/2,f,\chi)\!  
\begin{cases} L_{\mathfrak{c}}(2s-l/2,\chi \psi_S), & l \in 2\Z,\\1, & l \notin 2\Z,\end{cases}$$
$$L_{\mathfrak{c}}(s,\chi)=L(s,\chi)\prod_{\mathfrak{q}|\mathfrak{c}}(1-\chi(\mathfrak{q})N(\mathfrak{q})^{-s}),$$
\begin{equation}\label{definition of $G$}
G(\chi,2s-n-l/2) := \prod_{(\f{p},\f{c})=1} G_{\f{p}}(\chi,2s-n-l/2), 
\end{equation}
and the rest of notation is as in \cite[Section 6]{BMana}; in particular: $\b{\eta}_n=1_H\(\begin{smallmatrix} & -1_n\\ 1_n & \end{smallmatrix}\)$, and for $\sigma = (\sigma_{v}) \in\C^\a$ satisfying $\Re (\sigma_{v})\geq 0$ and $\Re (\sigma_v)+k_v-l/2>2n$ for each $v\in\a$,
$$c_{S,k}(\sigma) =\prod_{v \in\a}\(\pm \det(2S_{v})^{-n}2^{n(n+3)/2-4\sigma_{v}-nk_{v}}\pi^{n(n+1)/2}\frac{\Gamma_n(\sigma_{v}+k_{v}-\frac{l}{2}-\frac{n+1}{2})}{\Gamma_n(\sigma_{v}+k_{v}-\frac{l}{2})}\)$$ 
and $\Gamma_n(x):=\pi^{n(n-1)/4}\prod_{i=0}^{n-1}\Gamma (x-\frac{i}{2})$. Furthermore the character $\psi_S$ is the Hecke character corresponding to the quadratic extension
$F(\det(2S)^{1/2})/F$ if $l$ is odd and to the quadratic extension $F((-1)^{l/4}\det(S)^{1/2})/F$ if $l$ is even.
\end{thm}
Statement of the above theorem expresses a combination of equations (30) and (31) from \cite[Section 9]{BMana} before multiplying them by the factor $\mathcal{G}_{k-l/2,2n}(s-l/4)$. 

\begin{rem} 
In fact, the results proved in \cite{BMana} are more general than the ones presented above. Indeed, we worked with congruence subgroups of the form 
$$\tilde{\b{\Gamma}}:=\{ (\l,\mu,\kappa)g\in C[\f{o}, \f{b}^{-1}, \f{b}^{-1}]\rtimes \Gamma[\f{b}^{-1}\f{e},\f{bc}]:a_g-1_n\in M_n(\f{e})\} ,$$
where $\mathfrak{e}$ is an integral ideal such that $\f{c}\subset\f{e}$ and $\f{e}$ is prime to $\f{e}^{-1}\f{c}$; then
$$\b{\Gamma}_1(\mathfrak{c}) \subseteq \tilde{\b{\Gamma}} \subseteq \b{\Gamma}_0(\mathfrak{c}),$$
where the last group is obtained by setting $\mathfrak{e} = \mathfrak{o}$. However, in this paper we decided to work with $\f{e}=\f{c}$, because for simplicity reasons we restricted the proof of our main theorem to this case.
\end{rem}


\section{Arithmetic properties of Siegel-Jacobi modular forms}\label{sec:arithmetic}

In this section we will write $\b{f}$ (instead of $f$) for Siegel-Jacobi modular forms and use $f$ for other types of forms.

For a congruence subgroup $\b{\Gamma}$ of $\b{G}(F)$ and a subfield $K$ of $\mathbb{C}$ we define the set
\[
M_{k,S}^{n}(\b{\Gamma}, K) := \{ \b{f} \in M_{k,S}^{n}(\b{\Gamma}) : \b{f}(\tau,w) = \sum_{t,r} c(t,r) \mathbf{e}_{\a}(\tr(t\tau +\T{r}w)), \,\,c(t,r) \in K \} ;
\]
the subspace $S_{k,S}^{n}(\b{\Gamma}, K)$ consisting of cusp forms is defined in a similar way. Moreover, we write $M_{k,S}^{n}(K)$ for the union of all spaces $M_{k,S}^{n}(\b{\Gamma},K)$ for all 
congruence subgroups $\b{\Gamma}$.

For an element $\sigma \in Aut(\mathbb{C})$ and an element $k = (k_v) \in \mathbb{Z}^{\mathbf{a}}$ we define $k^{\sigma} := (k_{v\sigma}) \in \mathbb{Z}^{\mathbf{a}}$, where $v \sigma$ is the archimedean place corresponding to the embedding $ K \stackrel{\tau_v}{\hookrightarrow} \mathbb{C} \stackrel{\sigma}{\rightarrow} \mathbb{C}$, if $\tau_v$ is the embedding in $\mathbb{C}$ corresponding to the archimedean place $v$.  
\begin{prop} Let $k \in \mathbb{Z}^{\mathbf{a}}$, and let $\Phi$ be the Galois closure of $F$ in $\overline{\mathbb{Q}}$, and $\Phi_k$ the subfield of $\Phi$ such that
\[
Gal(\Phi/ \Phi_k) = \left\{ \sigma \in Gal(\Phi /F) : \,\,\,  k^{\sigma} = k \right\}.
\]
Then $M_{k,S}^{n}(\mathbb{C}) = M_{k,S}^{n}(\Phi_k) \otimes_{\Phi_k} \mathbb{C}.  $ 
\end{prop}
\begin{proof} If $F= \mathbb{Q}$, this is \cite[Proposition 3.8]{Sh78}. A careful examination of the proof \cite[page 60]{Sh78} shows that the proof is eventually reduced to the corresponding statement for Siegel modular forms of integral (if $l$ is even) or half-integral (if $l$ is odd) weight. However, in both cases the needed statement does generalize to the case of totally real fields, as it was established in \cite[Theorems 10.4 and 10.7]{Sh00}. 
\end{proof}

Given an $\b{f} \in M_{k,S}^{n}(\mathbb{C})$, we define
 \[
 \b{f}_*(\tau,w) := \mathbf{e}_{\a}(\tr(S w(\tau- \overline{\tau})^{-1}\transpose{w})) \b{f}(\tau,w)
 \]
and write $\mathbb{Q}^{ab}$ for the maximal abelian extension of $\mathbb{Q}$. Moreover, for $k \in \frac{1}{2} \mathbb{Z}^{\mathbf{a}}$ such that $k_v - \frac{1}{2} \in \mathbb{Z}$ for all $v \in \mathbf{a}$ we write $M_k^n$ for the space of Siegel modular forms of weight $k$, and of any congruence subgroup, and $M_k^n(K)$ for those with the property that all their Fourier coefficients at infinity  lie in $K$ (see for example \cite[Chapter 2]{Sh00} for a detailed study of these sets).

\begin{prop} \label{Hecke Operators preserve field of definition}Let $K$ be a field that contains $\Q^{ab}$ and $\Phi$ as above. Then  
\begin{enumerate}
\item $\b{f} \in M_{k,S}^{n}(K)$ if and only if $\b{f}_{*}(\tau, v \Omega_{\tau} ) \in M_{k}^{n}(K)$ for all $v \in M_{l,2n}(F)$, where $\Omega_{\tau} := \transpose{(\tau \,\,\,1_n)}$.
\item For any element $\gamma \in \Sp_n(F) \hookrightarrow \b{G}^n(F)$ and $\b{f} \in M_{k,S}^{n}(K)$,  we have 
\[
\b{f}|_{k,S} \gamma \in M_{k,S}^{n}(K).
\]
Moreover if $\b{f} \in M_{k,S}^n(\b{\Gamma},K)$, it follows that $\b{f}|T_{r} \in M_{k,S}^n(\b{\Gamma},K)$ for any $r \in Q(\f{c})$.
\end{enumerate}
\end{prop}

\begin{proof} If $F=\mathbb{Q}$, this is \cite[Proposition 3.2]{Sh78}. It is easy to see that the proof generalizes to the case of any totally real field. Indeed, the first part of the proof is a direct generalization of the argument used by Shimura. The second part requires the fact that the space $M_{k}^n(K)$ is stable under the action of elements in $\Sp_n(F)$, which is true for any totally real field, as it is proved in \cite[Theorem 10.7 (6)]{Sh00}. The last statement follows from the definition of the Hecke operator $T_{r}$.
\end{proof}

For a symmetric positive definite matrix $S\in Sym_l(F)$, $h \in M_{l,n}(F)$ and a lattice $L \subset M_{l,n}(F)$ we define the Jacobi theta series of characteristic $h$ by
\[
\Theta_{S,L,h}(\tau,w) = \sum_{x \in L} \mathbf{e_a}(\tr(S\(\frac{1}{2} (x + h)\tau\, \transpose{(x+h)} + w\,\transpose{(x+h)} \) )).
\]

\begin{thm}  \label{Structure of holomorphic Siegel-Jacobi}Assume that $n > 1$ or $F \neq \mathbb{Q}$, and let $K$ be any subfield of $\mathbb{C}$. Let $A \in \GL_l(F)$ be such that $A S \,\transpose{A} = \diag[s_1, \ldots, s_l ]$, and define the lattices $\Lambda_1 := A M_{l,n}(\mathfrak{o}) \subset M_{l,n}(F)$ and $\Lambda_2 := 2 \diag[s_1^{-1}, \ldots, s_l^{-1}] M_{l,n}(\mathfrak{o}) \subset M_{l,n}(F)$. Then there is an isomorphism
\[
\Phi: M^n_{k,S}(K) \cong \bigoplus_{h \in \Lambda_1/\Lambda_2} M^n_{k-l/2}(K)
\] 
given by $\b{f} \mapsto \left( f_{h}\right)_h$, where the $f_h \in M^n_{k-l/2}(K)$ are defined by the expression
\[
\b{f}(\tau,w) = \sum_{h \in \Lambda_1/\Lambda_2} f_h(\tau) \Theta_{2S,\Lambda_2,h}(\tau,w).
\]

Moreover, under the above isomorphism,
\[
   \Phi^{-1} \left( \bigoplus_{h \in \Lambda_1/\Lambda_2} S^n_{k-l/2}(K) \right) \subset S^n_{k,S}(K).
\]
\end{thm}

\begin{rem} \label{Condition $n=1$}We remark here that the assumption of $n>1$ or $F \neq \mathbb{Q}$ is needed to guarantee that the $f_h$'s are holomorphic at the cusps, which follows from the K\"{ocher} principle. However, even in the case of $F=\mathbb{Q}$ and $n=1$, if we take $\b{f}$ to be of trivial level, then the $f_h$'s are holomorphic at infinity (see for example \cite[page 59]{EZ}).
\end{rem}

\begin{proof}[Proof of Theorem \ref{Structure of holomorphic Siegel-Jacobi}] The first statement is \cite[Proposition 3.5]{Sh78} for $F=\mathbb{Q}$ and it easily generalizes to the case of any totally real field. We explain the statement about  cusp forms. 

Consider first expansions around the cusp at infinity. Fix $h\in \Lambda_1/\Lambda_2$ and let $f_h(\tau) = \sum_{t_2 >0} c(t_2) \mathbf{e}_{\a}(\tr(t_2\tau))$. It is known that Fourier coefficients $c(t_1,r)$ of a Jacobi theta series  
\[
\Theta_{2S,\Lambda_2,h}(\tau,w) = \sum_{t_1,r} c(t_1,r) \mathbf{e}_{\a}(\tr(t_1\tau)) \mathbf{e}_{\a}(\tr(\T{r}w))
\]
are nonzero only if $4t_1 = \transpose{r} S^{-1} r$ (see \cite[p. 210]{Z89}). Hence, the coefficients of
\[
f_h(\tau) \Theta_{2S,\Lambda_2,h}(\tau,w) = \sum_{t,r}\left( \sum_{t_1 + t_2 = t} c(t_1,r) c(t_2)\right) \mathbf{e}_{\a}(\tr(t\tau)) \mathbf{e}_{\a}(\tr(\T{r}w))
\] 
are nonzero only if $4t = 4(t_1+t_2) = \transpose{r} S^{-1}r + 4t_2 > \transpose{r} S^{-1}r$. This means that the function $f_h(\tau) \Theta_{2S,\Lambda_2,h}(\tau,w)$ satisfies cuspidality condition at infinity. 

Now let $\gamma$ be any element in $\Sp_n(F)$. The first statement in the Theorem states that for every $h_1 \in \Lambda_1/\Lambda_2$ there exist $f_{h_1,h_2}\in M_{k-l/2}^n(K),h_2\in\Lambda_1/\Lambda_2$, such that
\[
\Theta_{2S,\Lambda_2,h_1}|_{k,S}\gamma (\tau,w)= \sum_{h_2} f_{h_1,h_2}(\tau) \Theta_{2S, \Lambda_2,h_2} (\tau,w).
\]
Hence, for some cusp forms $f_{h_1}\in S_{k-l/2}^n(K)$,
\begin{align*}
\b{f}|_{k,S}\gamma (\tau,w) &:=  \sum_{h_1 } f_{h_1}|_{k}\gamma(\tau) \left(\sum_{h_2} f_{h_1,h_2}(\tau)\Theta_{2S, \Lambda_2,h_2} (\tau,w)\right)\\
&=\sum_{h_2} \left(\sum_{h_1} f_{h_1}|_{k}\gamma(\tau) f_{h_1,h_2}(\tau) \right) \Theta_{2S, \Lambda_2,h_2} (\tau,w).
\end{align*}
The same argument as used for the cusp at infinity implies that the functions $\b{f}|_{k,S}\gamma (\tau,w)$ and $\sum_{h_1} f_{h_1}|_{k}\gamma (\tau)f_{h_1,h_2}(\tau)$ are cuspidal. This finishes the proof.
\end{proof}

Note that the above theorem does not state that  $\Phi^{-1}\left( \bigoplus_{h \in \Lambda_1/\Lambda_2} S^n_{k-l/2}(K) \right) = S^n_{k,S}(K)$.
For this reason we make the following definition.

\noindent \textbf{Property A.} We say that a cusp form $\b{f} \in S_{k,S}^n(K)$ has the Property A if 
\[
\Phi(\b{f}) \in \bigoplus_{h \in \Lambda_1/\Lambda_2} S^n_{k-l/2}(K).
\]
 \textbf{Examples of Siegel-Jacobi forms that satisfy the Property A:} \label{ex:property A}
\begin{enumerate}
\item  Siegel-Jacobi forms over a field $F$ of class number one, and with trivial level, i.e. with $\mathfrak{c} = \mathfrak{o}$. Note that in this situation there is only one cusp. Then, keeping the notation as in the proof of the theorem above we need to verify that if $\b{f}(\tau,w) = \sum_{t,r} c_{\mathbf{f}}(t,r) \mathbf{e}_{\a}(\tr(t\tau)) \mathbf{e}_{\a}(\tr(\T{r}w))$ with $4t > \transpose{r}S^{-1}r$ whenever $c(t,r)\neq 0$, then all the $f_{h}$ have to be cuspidal. Observe first that if $h_1,h_2 \in \Lambda_1/\Lambda_2$ are different, $\Theta_{2S,\Lambda_2,h_1}(\tau,w) = \sum_{t,r} c_1(t,r) \mathbf{e}_{\a}(\tr(t\tau)) \mathbf{e}_{\a}(\tr(\T{r}w)),$ and $\Theta_{2S,\Lambda_2,h_2}(\tau,w) = \sum_{t,r} c_2(t,r) \mathbf{e}_{\a}(\tr(t\tau)) \mathbf{e}_{\a}(\tr(\T{r}w))$, then there is no $r$ such that at the same time $c_1(t,r) \neq 0$ and $c_2(t,r) \neq 0$. Indeed, if it was not the case then there would be $\lambda_1,\lambda_2 \in \Lambda_2$ such that $r = 2S (\lambda_1 + h_1)$ and $r = 2 S (\lambda_2 + h_2)$, that is, $\lambda_1 + h_1 = \lambda_2 + h_2$ or, equivalently, $h_1 - h_2 \in \Lambda_2$; contradiction. Hence, for any given $r$ there is a unique $h \in \Lambda_1/\Lambda_2$ such that $\Theta_{2S,\Lambda_2,h}$ has a nonzero coefficient $c(t,r)$. This means that there exists a unique $h$ such that $c_{\mathbf{f}}(t,r)$ is the Fourier coefficient of $f_h(\tau) \Theta_{2S,\Lambda_2,h}(\tau,w) = \sum_{t,r} \sum_{t_1 + t_2 = t} c(t_1,r) c(t_2) \mathbf{e}_{\a}(\tr(t\tau)) \mathbf{e}_{\a}(\tr(\T{r}w))$. But then $\transpose{r}S^{-1} r < 4t = 4(t_1+t_2) = \transpose{r}S^{-1}r + 4t_2$ and so $t_2 > 0$, which proves that $f_h$ is cuspidal. 

\item  Siegel-Jacobi forms of index $S$ such that $\det(2S) \in \mathfrak{o}^{\times}$, as in this case the lattices $\Lambda_1$ and $\Lambda_2$ from Theorem \ref{Structure of holomorphic Siegel-Jacobi} are equal.
\item Siegel-Jacobi forms of non-parallel weight, that is, if there exist distinct $v,v' \in \mathbf{a}$ such that $k_v \neq k_{v'}$. Indeed, in this case $M^n_{k-l/2}(K) = S^n_{k-l/2}(K)$ for all $h \in \Lambda_1/\Lambda_2$ (see \cite[Proposition 10.6]{Sh96}).
\end{enumerate}

\noindent \textbf{Remark:}\label{rem:property A} It is tempting to claim that Property A always holds. Actually the first example above suggests a possible way of establishing it for all Jacobi forms. However in the presence of more than one cusps, and of non-trivial index $[\Lambda_1:\Lambda_2]$ one needs to understand the behavior of the theta series $\Theta_{2S,\Lambda_2,h}(\tau,w)$ at all cusps, which seems to be quite hard in general. \newline

Let us now explain the significance of the Property A. Recall first that we have defined a Petersson inner product $<\b{f},\b{g}>$ when $\b{f},\b{g} \in M^n_{k,S}(K)$ and one of them, say, $\b{f}$ is cuspidal. If $\b{f}$ satisfies the Property A, then we claim that
\[
<\b{f},\b{g}> = N(\det(4S))^{-n/2}\sum_{h \in \Lambda_1/\Lambda_2} <f_h, g_h>.
\]
Indeed, as in \cite[Lemma 3.4]{Z89},  
\[
<\b{f},\b{g}> = N(\det(4S))^{-n/2}  vol(A)^{-1}\int_{A} \sum_{h \in \Lambda_1/\Lambda_2} f_h(\tau) \overline{g_h(\tau)} \det(\Im(\tau))^{k-l/2 -(n+1)}d\tau ,  
\]
where $A = \Gamma \back \mathbb{H}^{\mathbf{a}}_n$ and a congruence subgroup $\Gamma$ is deep enough. We obtain the claimed equality after exchanging the order of integration and summation. This can be done exactly because each $f_h$ is cuspidal, which makes each individual integral well defined.

\begin{lem} \label{projection to cuspidal part} Assume that $n>1$ or $F \neq \mathbb{Q}$ and that $\b{f} \in S_{k,S}^{n}(\overline{\mathbb{Q}})$ satisfies the Property A and one of the following two conditions holds:
\begin{enumerate}
\item[(i)] there exist $v,v' \in \mathbf{a}$ such that $k_v \neq k_{v'}$;
\item[(ii)] $k = \mu \mathbf{a} = (\mu,\ldots,\mu) \in \mathbb{Z}^{\mathbf{a}}$, with $\mu \in \mathbb{Z}$ depending on $n$ and $F$ in the following way:\\
\begin{tabular}{cccccccc}
$n >2$ & & $n =2, F=\mathbb{Q}$ & & $n =2, F\neq\mathbb{Q}$ & & $n=1$ & .\\
$\mu > 3n/2 +l/2$ & & $\mu > 3$ & & $\mu > 2$ & & $\mu \geq 1/2$ &
\end{tabular}
\end{enumerate} 
Then for any $\b{g}\in M_{k,S}^{n}(\overline{\mathbb{Q}})$ there exists $\widetilde{\b{g}} := \mathfrak{q}(\b{g}) \in S_{k,S}^{n}(\overline{\mathbb{Q}})$ such that 
\[
<\b{f},\b{g}> = <\b{f},\widetilde{\b{g}}>.
\]
\end{lem}

\begin{proof}  There is nothing to show in the case of non-parallel weight, since as it was mentioned above there is no (holomorphic) Eisenstein part in this case. In the parallel weight case, as $\b{f}$ has the Property A, 
$<\b{f},\b{g}> = N(\det(4S))^{-n/2} \sum_{h \in \Lambda_1/\Lambda_2} <f_h, g_h>$. 
It is known (see \cite[Theorem 27.14]{Sh00}) that under the assumptions stated in (ii), there exists a projection $\widetilde{\mathfrak{q}} : M_{k-l/2}^n(\overline{\mathbb{Q}}) \rightarrow S_{k-l/2}^n(\overline{\mathbb{Q}})$ such that $<f,g> = <f,\widetilde{g}>$ for all $f\in S_{k-l/2}^n(\overline{\mathbb{Q}})$, $g\in M_{k-l/2}^n(\overline{\mathbb{Q}})$. Then, if we put $\widetilde{g}_h := \widetilde{\mathfrak{q}}(g_h)$ for all $h \in \Lambda_1 /\Lambda_2$, we get 
\[
<\b{f},\b{g}> = N(\det(4S))^{-n/2} \sum_{h \in \Lambda_1/\Lambda_2} <f_h, g_h> = N(\det(4S))^{-n/2} \sum_{h \in \Lambda_1/\Lambda_2} <f_h, \widetilde{g}_h> .
\]
In particular, if we set $\widetilde{\b{g}} := \Phi^{-1}((\widetilde{g}_h)_h)$, we obtain the statement of the lemma.
\end{proof}

Now, fix a fractional ideal $\f{b}$ and an integral ideal $\f{c}$ of $F$, set $\b{\Gamma}:=\b{\Gamma}_1(\f{c})$ and let 
$\b{f} \in S_{k,S}^{n}(\b{\Gamma},\overline{\mathbb{Q}})$ be a non-zero Siegel-Jacobi form. Furthermore, assume that $\b{f}$ is an eigenfunction of the operators $T(\f{a})$ for all integral ideals $\f{a}$, and write $\b{f} | T(\f{a}) = \lambda(\f{a}) \b{f}$. We define the space
\[
V(\b{f}) := \{ \widetilde{\b{f}} \in S_{k,S}^{n}(\b{\Gamma},\overline{\mathbb{Q}}) :\widetilde{\b{f}} | T(\f{a}) = \lambda(\f{a}) \widetilde{\b{f}}\mbox{ for all }\f{a}\} .
\]

We are now ready to state the main theorem of this paper on algebraic properties of 
\[
\b{\Lambda}(s,\b{f},\chi) = L(2s-n-l/2,\b{f},\chi) \begin{cases} L_{\mathfrak{c}}(2s-l/2,\chi \psi_S)  & \hbox{if } l \in 2\mathbb{Z},\\
1 & \hbox{if } l \not \in 2\mathbb{Z}.
 \end{cases} 
\]

Recall that the $L$-function $L(s,\b{f},\chi)$ is defined only when the index matrix $S$ of $\b{f}$ satisfies the property $M_{\f{p}}^+$, as stated in Theorem \ref{Euler Product Representation}.

\begin{thm}\label{Main Theorem on algebraicity}
Assume $n >1$ or $F\neq \mathbb{Q}$, and that $S$ satisfies the property $M_{\f{p}}^+$ for all prime ideals $\mathfrak{p}$ of $F$ prime to $\mathfrak{c}$. Let $\chi$ be a Hecke character of $F$ such that $\chi_{\a}(x) =\sgn_{\a}(x)^k$, and $0 \neq \b{f} \in S_{k,S}^{n}(\b{\Gamma},\overline{\mathbb{Q}})$ an eigenfunction of all $T(\f{a})$. Set $\mu := \min_v{k_v}$ and assume that 
\begin{enumerate}
\item $\mu > 2n +l +1$,
\item  Property A holds for all $\widetilde{\b{f}} \in V(\b{f})$,
\item $k_v \equiv k_{v'} \mod{2}$ for all $v, v' \in \mathbf{a}$.
\end{enumerate}
Let $\sigma \in  \mathbb{Z}$ be such that 
\begin{enumerate}
\item[(i)] $2n+1 - (k_v - l/2)  \leq \sigma -l/2 \leq k_v -l/2$ for all $v \in \mathbf{a}$, 
\item[(ii)] $| \sigma - \frac{l}{2} - \frac{2n+1}{2} | + \frac{2n+1}{2} - (k_v - l/2) \in 2 \mathbb{Z}$ for all $v \in \mathbf{a}$,
\item[(iii)] $k_v > l/2 + n(1+k_v - l/2 -|\sigma - l/2 - (2n+1)/2| - (2n+1)/2)$ for all $v \in \mathbf{a}$,
\end{enumerate}
but exclude the cases
\begin{enumerate}
\item[(a)] $\sigma = n+1+ l/2$, $F = \mathbb{Q}$ and $\chi^2 \psi_i^2 = 1$ for some $\psi_i$,
\item[(b)] $\sigma = l/2$, $\mathfrak{c} = \mathfrak{o}$ and $\chi \psi_S \psi_i = 1$ for some $\psi_i$,
\item[(c)] $0 < \sigma - l/2 \leq n$, $\mathfrak{c} = \mathfrak{o}$ and $\chi^2 \psi_i^2 = 1$ for some $\psi_i$.
\item[(d)] $\sigma \leq l +n$ in case $F$ has class number larger than one.
\end{enumerate}

Under these conditions
\[
\frac{\b{\Lambda}(\sigma/2,\b{f},\chi)}{\pi^{e_{\sigma}} <\b{f},\b{f}>} \in \overline{\mathbb{Q}},
\]
where  
$$e_{\sigma} =  n \sum_{v \in \mathbf{a}} (k_v - l + \sigma) - de,\quad e := \begin{cases}  n^2 + n - \sigma + l/2, & \mbox{if } l  \in 2 \mathbb{Z} \mbox{ and } \sigma \geq 2n + l/2, \\ n^2, & \mbox{otherwise}. \end{cases} $$

\end{thm}

This theorem will be proved at the end of the next section.
First we need to introduce the notion of nearly holomorphic Siegel-Jacobi modular forms $N_{k,S}^{n,r}(\b{\Gamma})$ for $r \in \mathbb{Z}^{\mathbf{a}}$.
\section{Nearly holomorphic Siegel-Jacobi modular forms and algebraicity of special L-values}\label{sec:nhol}

\begin{defn} A $C^{\infty}$ function $\b{f} (\tau,w): \mathcal{H}_{n,l} \rightarrow \mathbb{C}$ is said to be a nearly holomorphic Siegel-Jacobi modular form (of weight $k$ and index $S$) for the congruence subgroup $\b{\Gamma}$ if 
\begin{enumerate}
\item $\b{f}$ is holomorphic with respect to the variable $w$ and nearly holomorphic with respect to the variable $\tau$, that is, $\b{f}$ as a function of $\tau$ belongs to the space $N^r(\mathbb{H}_n^d)$, $r \in \mathbb{Z}_+^d$, where the space  $N^r(\mathbb{H}_n^d)$ is defined in \cite[page 99]{Sh00};
\item $\b{f} |_{k,S} \gamma = \b{f}$ for all $\gamma \in \b{\Gamma}$.
\end{enumerate}
Actually one needs to also put the usual condition at the cusps when $n=1$ and $F = \mathbb{Q}$ but we will later restrict ourselves to the case of $n\neq 1$ or $F \neq \mathbb{Q}$ where this condition is automatic.
 
We denote this space by $N_{k,S}^{n,r}(\b{\Gamma})$ and write $N_{k,S}^{n,r} := \bigcup_{\b{\Gamma}} N_{k,S}^{n,r}(\b{\Gamma})$ for the space of all nearly holomorphic Siegel-Jacobi modular forms of weight $k$ and index $S$.
\end{defn}

It follows that $\b{f}\in N_{k,S}^{n,D}(\b{\Gamma})$ has a Fourier expansion of the form 
\begin{equation*}
\b{f}(\tau,w)=\sum_{\substack{t\in L\\ t\geq 0}}\sum_{r\in M} p_{t,r}((\Im(\tau_v)^{-1})_{v\in\a}) \mathbf{e}_{\mathbf{a}}(\tr(t\tau)) \mathbf{e}_{\mathbf{a}}(\tr(\T{r}w))
\end{equation*}
for some suitable lattices $L \subset Sym_n(F)$ and $M \subset M_{l,n}(F)$, where $p_{t,r}$ is a polynomial function on $Sym_n(F_\a)$ of total degree $D\in \mathbb{Z}_+^d$.

We note that if $\b{f} \in N_{k,S}^{n,r}$, then $\b{f}_*(\tau,v\,\Omega_{\tau}) \in N_{k}^{n,r}$, the space of nearly holomorphic Siegel modular forms, where recall $\Omega_{\tau} := \transpose{(\tau \,\,\,1_n)}$, and $v \in M_{l,2n}(F)$. The next theorem, which has been established in \cite{BMana}, extends Theorem \ref{Structure of holomorphic Siegel-Jacobi} to the nearly-holomorphic situation.

\begin{thm} \label{Structure of nearly holomorphic Siegel-Jacobi} Assume that $n > 1$ or $F \neq \mathbb{Q}$. Let $A \in \GL_l(F)$ be such that $A S\, \transpose{A} = \diag[s_1, \ldots, s_l ]$, and define the lattices $\Lambda_1 := A M_{l,n}(\mathfrak{o}) \subset M_{l,n}(F)$ and $\Lambda_2 := 2 \diag[s_1^{-1}, \ldots, s_l^{-1}] M_{l,n}(\mathfrak{o}) \subset M_{l,n}(F)$. Then there is an isomorphism
\[
\Phi: N^{n,r}_{k,S}\cong \bigoplus_{h \in \Lambda_1/\Lambda_2} N^{n,r}_{k-l/2}
\] 
given by $\b{f} \mapsto \left( f_{h}\right)_h$, where the $f_h \in N^{n,r}_{k-l/2}$ are defined by the expression
\[
\b{f}(\tau,w) = \sum_{h \in \Lambda_1/\Lambda_2} f_h(\tau) \Theta_{2S,\Lambda_2,h}(\tau,w).
\]
\end{thm}

The above theorem immediately implies the following.

\begin{cor}\label{cor:N_finite_dim} For a congruence subgroup $\b{\Gamma}$, $N_{k,S}^{n,r}(\b{\Gamma})$ is a finite dimensional $\mathbb{C}$ vector space.
\end{cor}
\begin{proof} The theorem above gives that $N_{k,S}^{n,r}(\b{\Gamma}) \hookrightarrow \bigoplus_h N^{n,r}_{k-l/2}(\Gamma_h)$ for some congruence subgroups $\Gamma_h$. The spaces $N^{n,r}_{k-l/2}(\Gamma_h)$ are known to be finite dimensional (see \cite[Lemma 14.3]{Sh00}). 
\end{proof}

Given an automorphism $\sigma \in Aut(\mathbb{C})$ and $\b{f} \in N_{k,S}^{n,r}$, we define
\[
\b{f}^{\sigma}(\tau,w) := \sum_{h \in \Lambda_1/\Lambda_2} f^{\sigma}_h(\tau) \Theta_{2S,\Lambda_2,h}(\tau,w),
\]
where $f_h \in N^{n,r}_{k-l/2}$, and $f_h^\sigma$ is defined as in \cite[page 117]{Sh00}.
Also, for a subfield $K$ of $\mathbb{C}$, define the space $N^{n,r}_{k,S}(K)$ to be the subspace of $N^{n,r}_{k,S}$ such that $\Phi(N^{n,r}_{k,S}(K))= \bigoplus_{h \in \Lambda_1/\Lambda_2} N^n_{k-l/2}(K)$. In particular, $\b{f} \in N_{k,S}^{n,r}$ belongs to $N^{n,r}_{k,S}(K)$ if and only if $\b{f}^\sigma = \b{f}$ for all $\sigma \in Aut(\C /K)$. Moreover, if $K$ contains the Galois closure of $F$ in $\overline{\Q}$ and $\Q^{ab}$, then $N_{k,S}^{n,r} = N_{k,S}^{n,r}(K) \otimes_K \mathbb{C}$ as the same statement holds for $N_{k-l/2}^{n,r}$. Similarly it follows that if $\b{f} \in N_{k,S}^{n,r}(\overline{\mathbb{Q}})$, then $\b{f}|_{k,S} \b{\gamma} \in N_{k,S}^{n,r}(\overline{\mathbb{Q}})$ for all $\b{\gamma} \in \b{G}(F)$. At this point we also remark that for an $\b{f} \in M_{k,S}^n$ the $\b{f}^c$ defined in Section \ref{sec:L-function} is nothing else than $\b{f}^{\rho}$ where $1 \neq \rho \in Gal(\mathbb{C}/\mathbb{R})$, i.e.\@, a complex conjugation. \newline

We now define a variant of the holomorphic projection in the Siegel-Jacobi case. We define a map $\f{p}\colon N_{k,S}^{n,r}(\overline{\Q}) \rightarrow M_{k,S}^n(\overline{\Q})$ whenever $k_v  - l/2> n + r_v$ for all $v \in \mathbf{a}$ by 
\[
\f{p}(\b{f}) := \f{p}\left(\sum_{h \in \Lambda_1/\Lambda_2} f_h(\tau) \Theta_{2S,\Lambda_2,h}(\tau,w)\right) := \sum_{h \in \Lambda_1/\Lambda_2} \widetilde{\f{p}}(f_h(\tau)) \Theta_{2S,\Lambda_2,h}(\tau,w),
\]
where $\widetilde{\f{p}}\colon N_{k-l/2}^{n,r}(\overline{\mathbb{Q}}) \rightarrow M_{k-l/2}^n(\overline{\mathbb{Q}})$ is the holomorphic projection operator defined for example in \cite[Chapter III, Section 15]{Sh00} and its algebraic properties are established in \cite[Lemma 28.2]{Sh00}. 

\begin{lem} \label{holomorphic_projection} Assume $n >1$ or $F\neq \mathbb{Q}$, and that $\b{f} \in S_{k,S}^n$ satisfies the Property A, and $k_v - l/2 > n + r_v$ for all $v \in \mathbf{a}$. Then for any $\b{g} \in N_{k,S}^{n,r}(\overline{\mathbb{Q}})$,
\[
<\b{f},\b{g}> = <\b{f},\f{p}(\b{g})>.
\]
\end{lem}
\begin{proof} This follows from the fact that the above property holds for nearly holomorphic Siegel modular forms, and the fact that the Property A allows us to write the Petersson inner product of Siegel-Jacobi forms as a sum of Petersson inner products of Siegel modular forms, in a similar way as we did in the proof of Lemma \ref{projection to cuspidal part}.  
\end{proof}

Let us write $F_1$ for the Hilbert class field extension of $F$ and denote by $\{\psi_i\}$ the ideal characters corresponding to the characters of $Gal(F_1/F)$. We can now state a theorem regarding the nearly holomorphicity of Siegel-type Jacobi Eisenstein series.  

\begin{thm}\label{algebraic properties of Eisenstein series} Consider the normalized Siegel-type Jacobi-Eisenstein series
\[
D(s):=D(z,s;k,\chi) :=  \Lambda_{k-l/2,\mathfrak{c}}^n(s-l/4,\chi \psi_S) E^n(z,s,\chi).
\]

Let $\mu \in \mathbb{Z}$ be such that 
\begin{enumerate}
\item[(i)] $n+1 - (k_v - l/2) \leq \mu-l/2 \leq k_v - l/2$ for all $v \in \mathbf{a}$, and
\item[(ii)] $| \mu -l/2 - \frac{n+1}{2}| + \frac{n+1}{2} - k_v + l/2 \in 2\mathbb{Z}$,
\end{enumerate}
but exclude the cases
\begin{enumerate}
\item[(a)] $\mu = \frac{n+2}{2}+ l/2$, $F = \mathbb{Q}$ and $\chi^2 \psi_i^2 = 1$ for some $\psi_i$,
\item[(b)] $\mu = l/2$, $\mathfrak{c} = \mathfrak{o}$ and $\chi \psi_S \psi_i = 1$ for some $\psi_i$,
\item[(c)] $0 < \mu - l/2 \leq n/2$, $\mathfrak{c} = \mathfrak{o}$ and $\chi^2 \psi_i^2 = 1$ for some $\psi_i$.
\item[(d)] $\mu \leq l +n$ if $F$ has class number larger than one.
\end{enumerate}
Then 
\[
D(\mu/2) \in \pi^{\beta} N_{k,S}^{n,r}(\overline{\mathbb{Q}}),
\]
where  
$$r = \begin{cases} {n(k-\mu+2)\over 2} & \mbox{if } \mu = {n+2\over 2} +{l\over 2}$, $F=\mathbb{Q}, \chi^2=1,\\
{k\over 2}-{l\over 4} & \mbox{if } n=1, \mu = 2 + {l\over 2}, F= \mathbb{Q}, \chi \psi_S = 1,\\
{n\over 2} k - {n\over 2}({l\over 2} +|\mu - {l\over 2} - {n+1\over 2}| + {n+1\over 2}) \mathbf{a} & \mbox{otherwise} \end{cases}$$
and by $x\mathbf{a}$ we understand a tuple $(x,\ldots,x)$ of length $d$.
Moreover,
$\beta = {n\over 2} \sum_{v \in \mathbf{a}} (k_v - l + \mu) - de$,
where 
$$e :=\begin{cases} [{(n+1)^2\over 4}] - \mu + {l\over 2} & \mbox{if } 2\mu -l+n\in 2\mathbb{Z}, \mu \geq n + {l\over 2},\\
[{n^2\over 4}] & \mbox{otherwise} .\end{cases} $$
\end{thm}

\begin{proof}The proof is similar to the proof of Theorem 8.3 in \cite{BMana}, where the analytic properties of the series $D(s)$ were established. Indeed, the relation between the Jacobi Eisenstein series and the classical Siegel Eisenstein series (for details see \cite[Section 8]{BMana}) allows us to read off the near holomorphy of $D(s)$ from the near holomorphy of the latter series, which was established in \cite[Theorem 17.9]{Sh00}.  

To be more precise, in \cite{BMana} (page 45 there, just before Theorem 8.2) we related the Jacobi Eisenstein series $E^n(z,s,\chi)$ to a sum involving the Siegel Eisenstein series $E^n(\tau,s-l/4;\chi \psi_S \psi_i)$ (with the notation as in \cite{BMana}), where $\psi_i$'s vary over all the finite unramified characters of $F$, that is the characters corresponding to the Hilbert class group of $F$. However, the normalizing factor for the Jacobi Eisenstein series $D(s)$ is $\Lambda_{k-l/2,\mathfrak{c}}^n(s-l/4,\chi \psi_S)$ whereas for the various Siegel Eisenstein series is $\Lambda_{k-l/2,\mathfrak{c}}^n(s-l/4,\chi \psi_S \psi_i)$, which of course depends on the characters $\psi_i$. 

Therefore we need to assure that the series
\[
\frac{\Lambda_{k-l/2,\mathfrak{c}}^n(\mu/2-l/4,\chi \psi_S)}{\Lambda_{k-l/2,\mathfrak{c}}^n(\mu/2-l/4,\chi \psi_S \psi_i)} \Lambda_{k-l/2,\mathfrak{c}}^n(\mu/2-l/4,\chi \psi_S \psi_i)E^n(\tau,\mu/2-l/4;\chi \psi_S \psi_i)
\]
has the same algebraic properties, i.e., whether it is defined over $\overline{\mathbb{Q}}$, as the normalized series 
$$\Lambda_{k-l/2,\mathfrak{c}}^n(\mu/2-l/4,\chi \psi_S \psi_i)E^n(\tau,\mu/2-l/4;\chi \psi_S \psi_i),$$
which is known (see \cite[Theorem 17.9]{Sh00}) to be in  $\pi^{\beta} N_{k,S}^{n,r}(\overline{\mathbb{Q}})$ with the choices of $\beta$ and $r$ as in the theorem.
Of course, we need to exclude the cases where the factor $\frac{\Lambda_{k-l/2,\mathfrak{c}}^n(\mu/2-l/4,\chi \psi_S)}{\Lambda_{k-l/2,\mathfrak{c}}^n(\mu/2-l/4,\chi \psi_S \psi_i)}$ has a pole. 
Then it remains to check whether
 \begin{equation}\label{eq:Lambda_alg}
 \frac{\Lambda_{k-l/2,\mathfrak{c}}^n(\mu/2-l/4,\chi \psi_S)}{\Lambda_{k-l/2,\mathfrak{c}}^n(\mu/2-l/4,\chi \psi_S \psi_i)} \in \overline{\mathbb{Q}}.
\tag{$\ast$}
 \end{equation} 
If the class number of $F$ is one, \eqref{eq:Lambda_alg} holds trivially. If it is greater than one, this should follow from the general Beilinson conjectures for motives associated to finite Hecke characters over totally real fields (see for example \cite{Scholl}). However since this is not known in general, in the proof below we are forced to set then the condition $\mu > n+l$. 

Recall first that for a finite Hecke character $\phi$ of $F$, we defined
$$\Lambda^{n}_{k-l/2,\mathfrak{c}}(\mu/2 - l /4,\phi) = \begin{cases} L_{\mathfrak{c}}(\mu - l/2,\phi) \prod_{j=1}^{[n/2]} L_{\mathfrak{c}}(2\mu-l-2j,\phi^{2}) & \hbox{if } l \in 2\mathbb{Z},\\
 \prod_{j=1}^{[(n+1)/2]} L_{\mathfrak{c}}(2\mu-l-2j+1,\phi^{2}) & \hbox{if } l \notin 2\mathbb{Z}.
 \end{cases}$$
We prove the statement in \eqref{eq:Lambda_alg} term by term for quotients of the corresponding factors in the product above. For this we use the fact stated in \cite[Theorem A6.5]{Sh96}: if $\phi$ is a finite Hecke character of infinite type of the form $\prod_{v \in \mathbf{a}} \(\frac{x_v}{|x_v|}\)^{\ell}$ for some positive integer $\ell$ (for us, since $F$ is totally real, only the parity of $\ell$ matters), then $L_{\mathfrak{c}}(\ell,\phi) \in \pi^{\ell d} \,\,\overline{\mathbb{Q}}$. Now observe that the characters $\chi \psi_S$ and $\chi \psi_S \psi_i$ (for each $i$) have the same infinite type since the characters $\psi_i$ are even at all infinite places due to the fact that they correspond to the Hilbert class field extension of $F$ (unramified also at infinity). In particular this establishes that for all $i$,
$$\frac{\prod_{j=1}^{[(n+\delta(l))/2]} L_{\mathfrak{c}}(2\mu-l-2j+\delta(l),\chi^{2})}{\prod_{j=1}^{[(n+\delta(l))/2]} L_{\mathfrak{c}}(2\mu-l-2j+\delta(l),\chi^{2} \psi_i^2)} \in \overline{\mathbb{Q}}\,,$$ 
where $\delta(l)$ is zero or one depending whether $l$ is even or odd, and where we have used the fact that $\psi_S^2=1$.

Now it remains to show that for $l$ even and $\mu > n+l$, $\frac{L_{\mathfrak{c}}(\mu - l/2,\chi \psi_S)}{L_{\mathfrak{c}}(\mu - l/2,\chi \psi_S \psi_i)} \in \overline{\mathbb{Q}}$ for all $i$. This condition on $\mu$ together with the assumption $(ii)$ imply that $\mu \equiv k_v \pmod{2}$ for all $v \in \mathbf{a}$. 
Furthermore, since $\psi_S$ is the Hecke character corresponding to the quadratic extension
$F((-1)^{l/4}\det(S)^{1/2})/F$ (see Theorem \ref{thm:dm_identity}), $l/2$ has the same parity as the character $\psi_S$. Hence the parity of $\mu-l/2$ coincides with the parity of the characters $\chi\psi_S$ and $\chi \psi_S \psi_i$ for all $i$.
\end{proof}

\begin{lem} \label{diagonal restriction preserves algebraicity} Consider the embedding 
\[
\Delta : \mathcal{H}_{n,l} \times \mathcal{H}_{n,l} \hookrightarrow \mathcal{H}_{2n,l},\,\,\,(\tau_1,w_1) \times (\tau_2,w_2) \mapsto (\diag[\tau_1, \tau_2], (w_1 \,\,w_2)).
\]
Then the pullback
\[
\Delta^*\left( N^{2n,r}_{k,S}(\overline{\Q}) \right) \subset N^{n,r}_{k,S}(\overline{\Q}) \otimes_{\overline{\Q}}   N^{n,r}_{k,S}(\overline{\Q}). 
\]
\end{lem}
\begin{proof} The proof of this lemma is identical to the Siegel modular form case (see \cite[Lemma 24.11]{Sh00}). Let $\b{f} \in N^{2n,r}_{k,S}(\b{\Gamma}^{2n}, \overline{\mathbb{Q}})$ for a sufficiently deep congruence subgroup $\b{\Gamma}^{2n}$. By \cite[Lemma 10.6]{BMana}, we have that
\[
\b{g}(z_1,z_2) = \sum_i \b{g}_i(z_1) \b{h}_i(z_2),
\]
where $\b{g}_i \in N^{n,r}_{k,S}$, and $\b{h}_i \in N_{k,S}^{n,r}$, where the latter form a basis of $N_{k,S}^{n,r}$. Using now the fact that $N_{k,S}^{n,r} = N_{k,S}^{n,r}(\overline{\mathbb{Q}}) \otimes_{\overline{\mathbb{Q}}} \mathbb{C}$, we may take $\b{h}_i \in N_{k,S}^{n,r}(\overline{\mathbb{Q}})$. Now, for any $\sigma \in Aut(\mathbb{C}/\overline{\mathbb{Q}})$, 
\[
\b{g}(z_1,z_2) = \b{g}^{\sigma}(z_1,z_2) = \sum_i \b{g}^{\sigma}_i(z_1) \b{h}^{\sigma}_i(z_2) = \sum_i \b{g}^{\sigma}_i(z_1) \b{h}_i(z_2).
\] 
Hence, $\b{g}^{\sigma}_i(z_2) = \b{g}_i(z_2)$ for all $\sigma \in  Aut(\mathbb{C}/\overline{\mathbb{Q}})$, and thus $\b{g}_i \in N^{n,r}_{k,S}(\overline{\mathbb{Q}})$.
\end{proof}

We can now establish a theorem which is the key result towards Theorem \ref{Main Theorem on algebraicity}. We note that this kind of results are well known for the case of Siegel and Hermitian modular forms, see for example \cite[Theorem 28.5 and Theorem 29.5]{Sh00}, and our proof is inspired by the proofs of these theorems. We would like to emphasise that, as in the case of Siegel and Hermitian modular forms, a vital component of the proof is a result on the algebraic splitting of the cuspidal part from the Eisenstein part. For the case of Siegel-Jacobi modular forms this is provided by Lemma \ref{projection to cuspidal part} above, and it is precisely this part that makes the Property A so important for our result.  

\begin{thm} \label{ratio of inner products}
Assume $n >1$ or $F\neq \mathbb{Q}$. Let $0 \neq \b{f} \in S_{k,S}^{n}(\b{\Gamma},\overline{\mathbb{Q}})$, with $S$ satisfying the condition $M_{\f{p}}^+$ for all prime ideals $\f{p}$ of $F$ coprime to $\f{c}$, be an eigenfunction of $T(\mathfrak{a})$ for all integral ideals $\mathfrak{a}$ with $(\f{a},\f{c})=1$. Define $\mu := \min_{v \in \mathbf{a}}{\{k_v\}}$ and assume that
\begin{enumerate}
\item $\mu > 2n +l +1$,
\item Property A holds for all $\widetilde{\b{f}} \in V(\b{f})$,
\item $k_v \equiv k_{v'} \mod{2}$ for all $v, v' \in \mathbf{a}$.
\item $k_v  > l/2 + n(1+ k_v - \mu)$ for all $v \in \mathbf{a}$. 
\end{enumerate}
Then for any $\b{g} \in M_{k,S}^n(\overline{\mathbb{Q}})$,
\[
\frac{<\b{f},\b{g}>}{ <\b{f}, \b{f}>} \in \overline{\mathbb{Q}}.
\]
\end{thm}
\begin{proof} By Lemma \ref{projection to cuspidal part} it suffices to prove this theorem for $\b{g} \in S_{k,S}^n(\overline{\mathbb{Q}})$. Furthermore, as it was shown in \cite[Section 7.4]{BMana}, the Hecke operators are normal and Proposition \ref{Hecke Operators preserve field of definition}  states that the Hecke operators $T(\f{a})$ preserve $S_{k,S}^{n}(\b{\Gamma},\overline{\mathbb{Q}})$. That is, we have a decomposition
\[
S_{k,S}^{n}(\b{\Gamma},\overline{\Q}) = V(\b{f}) \oplus \b{U},
\]
where $\b{U}$ is a $\overline{\mathbb{Q}}$-vector space orthogonal to $V(\b{f})$. Therefore, without loss of generality, we may assume that $\b{g} \in V(\b{f})$. 

Now consider a character $\chi$ of conductor $\f{f}_{\chi} \neq \mathfrak{o}$ such that $\chi_{\a}(x) = \sgn_{\a}(x)^{k}$, $\chi^2 \neq 1$ and $G(\chi,\mu-n-l/2) \in \overline{\mathbb{Q}}^{\times}$, where 
$G(\chi,\mu-n-l/2)$ is as in Theorem \ref{thm:dm_identity}, equation \eqref{definition of $G$}. 
The existence of such a character follows from the fact that $G(\chi,2s-n-l/2)$ is the ratio of products of finitely many Euler polynomials.  

We recall that if $\tilde{\b{f}} \in V(\b{f})$, then so is $\tilde{\b{f}}^c \in V(\b{f})$ and their $L$-functions agree. In particular, up to some non-zero algebraic number, the identity \eqref{eq:dm_identity} in Theorem \ref{thm:dm_identity} becomes:
\begin{multline*}
\Lambda_{k-l/2,\mathfrak{c}}^{2n}(\mu/2-l/4,\chi \psi_S)  vol(A)  <(E|_{k,S} \b{\rho}) (\diag[z_1,z_2],\mu/2;\chi),(\tilde{\b{f}}^c|_{k,S}\b{\eta}_n)^c(z_2)>\\
={^{\overline{\mathbb{Q}}}}^{\times} c_{S,k}(\mu/2-k/2) \b{\Lambda}(\mu/2,\b{f},\chi) \tilde{\b{f}}^c(z_1), 
\end{multline*}
where, recall, $\b{\eta}_n=1_H\(\begin{smallmatrix} & -1_n\\ 1_n & \end{smallmatrix}\)$. Since we take $\mu > 2n +l +1$, and $\mu$ is equal to the smallest of the $k_v$'s, it follows that for all $v \in \mathbf{a}$:
\begin{enumerate}
\item $2n+1 - (k_v - l/2) \leq \mu-l/2 \leq k_v - l/2$,
\item $\mu   - k_v  \in 2\mathbb{Z}$;
\end{enumerate}

the last condition is possible because all $k_v$ have the same parity. Thanks to the above choice of $\mu$ we can apply Theorem \ref{algebraic properties of Eisenstein series} (note that the excluded cases there do not apply for  $\mu$ larger than $2n+l + 1$) so that 
\[
\Lambda_{k-l/2,\mathfrak{c}}^{2n}(\mu/2-l/4,\chi \psi_S) E^{2n}(z,\mu/2; \chi) \in \pi^{\beta} N_{k,S}^{2n,r}(\overline{\mathbb{Q}})
\]
for some $\beta \in \mathbb{N}$ and 
$$r = n (k  - \mu\a ).$$

The same then holds for 
\[
\Lambda_{k-l/2,\mathfrak{c}}^{2n}(\mu/2-l/4,\chi \psi_S) E^{2n}(z,\mu/2; \chi)|_{k,S} \b{\rho}.
\]
Indeed if we set
$$\mathcal{E}(z,\mu/2)= \Lambda_{k-l/2,\mathfrak{c}}(\mu/2-l/4,\chi \psi_S) E^{2n}(z,\mu/2;\chi) \in N_{k,S}^{2n,r}(\overline{\mathbb{Q}}),$$
then by Theorem \ref{Structure of nearly holomorphic Siegel-Jacobi}, 
\[
\mathcal{E}(z,\mu/2)  = \sum_{h \in \Lambda_1/\Lambda_2} E_h(\tau) \,\,\Theta_{2S,\Lambda_2,h}(\tau,w)
\]
for some $E_h \in N^{2n,r}_{k-l/2}$. Further, since $\b{\rho} = 1_H \rho$ with $\rho \in Sp_{2n}(F)$, we can write
 \[
 \mathcal{E}(z,\mu/2) |_{k,S} \b{\rho}= \sum_{h \in \Lambda_1/\Lambda_2} E_h(\tau) |_{k} \rho\, \Theta_{2S,\Lambda_2,h}(\tau,w)|_{k,S}\b{\rho},
 \]
where by \cite[page 153]{Sh94b} we have $E_h|_k\rho \in N^{2n,r}_{k-l/2}$. In this way $\mathcal{E}(z,\mu/2) |_{k,S} \b{\rho} \in N^{2n,r}_{k,S}$.

In particular, we can now conclude that,
\[
\pi^{-\beta}\Lambda_{k-l/2,\mathfrak{c}}^{2n}(\mu/2-l/4,\chi \psi_S) (E|_{k,S} \b{\rho}) (\diag[z_1,z_2],\mu/2; \chi) = \sum_i \b{f}_i(z_1) \b{g}_i(z_2),
\] 
where $\b{f}_i,\b{g}_i \in N^{n,r}_{k,S}(\overline{\mathbb{Q}})$ by Lemma \ref{diagonal restriction preserves algebraicity}. Moreover, $vol(A) = \pi^{d_0} \mathbb{Q}^{\times}$, where $d_0$ is the dimension of $\mathbb{H}_n^d$ since the volume of the Heisenberg part is normalized to one. Furthermore, 
\[
c_{S,k}(\mu/2-k/2) \in \pi^{\delta} \overline{\mathbb{Q}}^{\times},\,\,\,\delta \in \frac{1}{2}\mathbb{Z}.
\]
Altogether we obtain
\[
 \sum_i \b{f}_i(z_1) <\b{g}_i(z_2), \b{w}(z_2) > ={^{\overline{\mathbb{Q}}}}^{\times} \pi^{\delta- d_0 +\beta} \b{\Lambda}(\mu/2,\b{f},\chi) \tilde{\b{f}}^c(z_1),
\]
where $\b{w}:= (\tilde{\b{f}}^c|_{k,S}\b{\eta}_n)^c=\tilde{\b{f}}|_{k,S}\b{\eta}^{-1}_n\in S_{k,S}^n(\overline{\mathbb{Q}})$. Considering the Fourier expansion of $\b{f}_i$'s and $\tilde{\b{f}}^c$, and comparing any $(r,t)$ coefficients for which $c(r,t;\tilde{\b{f}}^c) = \overline{c(r,t;\tilde{\b{f}})} \neq 0$, we find that 
\[
<\sum_i \alpha_{i,r,t} \b{g}_i(z_2),\b{w}(z_2)> = {^{\overline{\mathbb{Q}}}}^{\times} \pi^{\delta- d_0 +\beta} \b{\Lambda}(\mu/2,\b{f},\chi) \neq 0,
\]
for some $\alpha_{i,r,t} \in \overline{\mathbb{Q}}$, where the non-vanishing follows from \eqref{eq:non-vanishing of L-values}, a corollary to Theorem \ref{Euler Product Representation}. Setting $\b{h}_{r,t}(z_2) := \sum_i \alpha_{i,r,t} \b{g}_i(z_2) \in N^{n,r}_{k,S}(\overline{\mathbb{Q}})$, we obtain 
\[
<\b{h}_{r,t}(z_2), \b{w}(z_2)>={^{\overline{\mathbb{Q}}}}^{\times} \pi^{\delta- d_0 +\beta} \b{\Lambda}(\mu/2,\b{f},\chi) \neq 0,
\]
or,
\[
<\mathfrak{p}^0(\b{h}_{r,t}|_{k,S}\b{\eta}_n)(z_2), \tilde{\b{f}}(z_2)>= {^{\overline{\mathbb{Q}}}}^{\times} \pi^{\delta- d_0 +\beta} \b{\Lambda}(\mu/2,\b{f},\chi) \neq 0,\,\,\,\,(*)
\]
where 
$$\f{p}^0 :=\begin{cases} \f{p}, & k\mbox{ not parallel},\\
\f{q} \circ \f{p}, & k\,\,\mbox{parallel}, \end{cases}$$
and we have used the assumptions (2) and (4) in the statement of the theorem  in order to be able to apply Lemma \ref{holomorphic_projection}. \newline

Since $\tilde{\b{f}} \in V(\b{f})$ was arbitrary, the forms $\widetilde{\b{h}}_{r,t} := \mathfrak{p}^0(\b{h}_{r,t}|_{k,S}\b{\eta}_n) \in S_{k,S}^n(\overline{\mathbb{Q}})$  (or rather their projections to $V(\b{f})$) for the various $(r,t)$ span the space $V(\b{f})$ over $\overline{\mathbb{Q}}$. Indeed, if we denote by $\mathcal{S} \subset V(\b{f})$ the $\overline{\mathbb{Q}}$ vector space spanned by the projections of $\widetilde{\b{h}}_{r,t}$ to $V(\b{f})$ and if there exists an $\tilde{\b{f}} \in V(\b{f})$ which is not in $\mathcal{S}$ then there is a form, say $\tilde{\b{f}_1} \in V(\b{f})$ which is orthogonal to $\mathcal{S}$. But then this would imply by $(*)$ above that 
$c(r,t;\tilde{\b{f}_1}^c) = 0$ for all $(r,t)$. In particular $\tilde{\b{f}_1} = 0$. Hence indeed $\widetilde{\b{h}}_{r,t}$ span $V(\b{f})$.
Moreover we have
\[
<\widetilde{\b{h}}_{r,t},\tilde{\b{f}}> \in \pi^{\delta- d_0 +\beta} \b{\Lambda}(\mu/2,\b{f},\chi) \overline{\mathbb{Q}}^{\times}.
\]  
That is, for any $\b{g} \in V(\b{f})$ we have $<\b{g},\b{f}> \in \pi^{\delta- d_0 +\beta} \b{\Lambda}(\mu/2,\b{f},\chi) \overline{\mathbb{Q}}^{\times}$. In particular, the same holds for $\b{g} = \b{f}$, and that concludes the proof.
\end{proof}
We are now ready to give the proof of Theorem \ref{Main Theorem on algebraicity}.

\begin{proof}[Proof of Theorem \ref{Main Theorem on algebraicity}] We follow the same steps as in the proof of Theorem \ref{ratio of inner products} but this time we apply Theorem \ref{algebraic properties of Eisenstein series} with selecting $\mu: = \sigma$ there. Note that the restrictions on $\sigma$ in conditions $(i)$ and $(ii)$ are the ones that make the corresponding Eisenstein series nearly holomorphic (compare with conditions in Theorem \ref{algebraic properties of Eisenstein series}). Condition $(iii)$ allows us to use the holomorphic projection operator. Finally, the restrictions on the minimal weight are set so that we can apply the above Theorem \ref{ratio of inner products} on the algebraicity of the ratio of the Petersson inner products. \newline

In exactly the same way as above we obtain
\[
<\b{h}_{r,t}(z_2), \b{f}(z_2)>={^{\overline{\mathbb{Q}}}}^{\times} \pi^{\delta- d_0 +\beta} \b{\Lambda}(\sigma/2,\b{f},\chi),
\]
for some $\b{h}_{r,t} \in \b{N}_{k,S}^n(\overline{\mathbb{Q}})$. In particular we obtain,
\[
<\mathfrak{p}^0(\b{h}_{r,t}(z_2)), \b{f}(z_2)>= {^{\overline{\mathbb{Q}}}}^{\times} \pi^{\delta- d_0 +\beta} \b{\Lambda}(\sigma/2,\b{f},\chi),
\]
where 
$$\f{p}^0 :=\begin{cases} \f{p}, & k\mbox{ not parallel},\\
\f{q} \circ \f{p}, & k\,\,\mbox{parallel}, \end{cases}$$ and $\mathfrak{p}^0(\b{h}_{r,t}(z_2)) \in M_{k,S}$.

Thanks to Theorem \ref{ratio of inner products} the proof will be finished after dividing the above equality by $<\b{f},\b{f}>$ if we make the powers of $\pi$ precise. Recall that
\begin{multline*}
c_{S,k}(\sigma/2 - k/2) ={^{\overline{\mathbb{Q}}}}^{\times} \pi^{dn(n+1)/2} \prod_{v \in \mathbf{a}} \frac{\Gamma_n(\sigma/2 + k_v -l/2 - (n+1)/2)}{\Gamma_n(\sigma/2 +k_v - l/2)}\\
={^{\overline{\mathbb{Q}}}}^{\times} \pi^{dn(n+1)/2} \prod_{v \in \mathbf{a}} \frac{\prod_{i=0}^{n-1}\Gamma(\sigma/2 + k_v -l/2 - (n+1)/2 - i/2)}{\prod_{i=0}^{n-1}\Gamma(\sigma/2 +k_v - l/2 - i/2)} ={^{\overline{\mathbb{Q}}}}^{\times} \pi^{dn(n+1)/2}. 
\end{multline*}
Hence, $\delta = dn(n+1)/2$. However, this is also equal to the dimension of the space $\mathbb{H}_n^d$, which we denoted by $d_0$. We are then left only with $\beta$, which is provided by Theorem \ref{algebraic properties of Eisenstein series}; namely,
\[
\beta =  n \sum_{v \in \mathbf{a}} (k_v - l + \sigma) - de, 
\]
where $e := n^2 + n - \sigma + l/2$ if $2 \sigma -l  \in 2 \mathbb{Z}$ and $\sigma \geq 2n + l/2$, and $e:= n^2$ otherwise. This concludes the proof of the theorem.
\end{proof}

\section{Poincar\'e series of exponential type and holomorphic projection}\label{sec:Sturm}

Our results in the previous section were obtained under the assumption of Property A, which allowed us to obtain Lemma \ref{projection to cuspidal part} and Lemma \ref{holomorphic_projection} above. The first one allowed us to split the Eisenstein series from the cuspidal part while preserving the algebraicity of the coefficients (the operator $\mathfrak{q}$ above), and the second to define a projection from nearly holomorphic Siegel-Jacobi forms to modular forms via the known projection from Siegel modular forms of integral (when $\ell \in 2 \mathbb{Z}$) or half-integral (when $\ell \in \mathbb{Z}$) weight (the operator $\mathfrak{p}$ above). As we indicated above, Property A is known to hold in many cases, as for example in the case of non-parallel weight. However in general, it may not be so easy to check in practice. For this reason in this section we develop a different approach to the holomorphic projection, which does not rely on Property A, and projects the given nearly holomorphic form directly to a cusp form. This new approach is modelled on the one developed by Sturm \cite{St} in the case of Siegel modular forms; and as it is there, Poincar\'e series play a central role. In the first part of this section we study properties of such series in the Jacobi setting.   

In the next section we will apply the results of this section to obtain algebraicity results of the special $L$-values without assuming Property A. For simplicity we will restrict ourselves to the case of $F = \mathbb{Q}$, even though the results of this section should generalize to the case of totally real fields. \newline 

We let $\b{\Gamma}=H(\Z)\rtimes\Gamma$, equipped with a homomorphism $\chi$, where $\Gamma$ is a subgroup of $\Sp_n(\Z)$ of finite index. We denote by $\f{M}^n_{k,S}(\b{\Gamma} ,\chi)$ the space of $C^{\infty}$ (smooth) functions $f:\H_{n,l}\to\C$ such that $f|_{k,S}\, \b{g} =\chi(\b{g})f$ for every $\b{g} \in \b{\Gamma}$. We further let $\lag$ be the smallest positive integer for which
$$\b{\Gamma}_{\infty}:=\{ (0,\mu,0)\begin{pmatrix} 1_n & b\\ & 1_n\end{pmatrix} : \mu\in M_{l,n}(\Z), b\in\lag Sym_n(\Z)\}$$
is contained in the kernel of $\chi$. It follows that $f$ has a Fourier epxansion of the form
\begin{equation}\label{eq:Fexp-in-y}
f(\tau,w)=\sum_{t\in L}\sum_{r\in M_{l,n}(\Z)} A_{t,r}(\Im(\tau)) e\(\frac{1}{\lag}\tr(t\Re(\tau))\) e(\tr(\T{r}w)),
\end{equation}
where $L:=\{ t\in \frac12 Sym_n(\Z): t_{ii}\in\Z\mbox{ for all }i=1,\ldots ,n\}$ and $A_{t,r}$ are $C^\infty$ functions on 
$$Y_n:=\{ y\in M_n(\R): \T{y}=y, y>0\} .$$
Throughout this section we will write $\tau= x+iy$, $w=u+iv$ with $x, y, u, v$ having real entries.

Further, for a positive definite matrix $t\in L$ and $r\in M_{l,n}(\Z)$ such that $4t-\lag S^{-1}[r]>0$, we define the $(t,r)$-th Siegel-Jacobi Poincar\'e series
\begin{equation}
P_{t,r}(\tau,w):=\sum_{\b{g}\in Z_l\b{\Gamma}_{\infty}\back\b{\Gamma}}\overline{\chi(\b{g})}e\(\frac{1}{\lag}\tr(t\tau)\) e(\tr(\T{r}w))|_{k,S}\,\b{g} ,
\end{equation}
where $Z_l:=\{ \pm (0,0,\kappa)1_{2n}:\kappa\in Sym_l(\Z)\} \cap\b{\Gamma}$ is the subgroup of $\b{\Gamma}$ which acts trivially on $\H_{n,l}$. 

\begin{prop}\label{prop:convergence}
The Poincar\'e series $P_{t,r}$ converges absolutely and locally uniformly on $\H_{l,n}$ for $k>n+l+1$. If $k>2n+l$, then $P_{t,r}\in S^n_{k,S}(\b{\Gamma},\chi)$.
\end{prop}

In order to prove this proposition, and also for further calculations, we need the following two identities:

\begin{itemize}
\item For positive integers $k,n$ such that $k>(n-1)/2$ and for $\tau\in \HH_n$,
\begin{equation}\label{eq:int-det}
\int_{Y_n}\det(t)^{k-(n+1)/2}e^{-\tr(t\tau)}dt=\Gamma_n(k)\det (\tau)^{-k},
\end{equation}
where $dt=\prod_{1\leq i\leq j\leq n} dt_{ij}$ and $\Gamma_n(k):=\pi^{n(n-1)/4}\prod_{j=0}^{n-1}\Gamma(k-j/2)$.
\item For a positive definite matrix $S\in Sym_l(\R)$, $R\in M_{n,l}(\R)$, $A\in Sym_n(\C)$ and $a\in\C^{\times}$, 
\begin{multline}\label{eq:cool-id}
\int_{M_{l,n}(\R)} \exp(a\tr(-S[X]A+RXA)) dX\\
=(\det A)^{-l/2}\({\pi\over a}\)^{nl/2}(\det S)^{-n/2}\exp \({a\over 4}\tr(S^{-1}[\T{R}]A)\) .
\end{multline}
\end{itemize}
The first identity is derived in \cite[Lemma 6.2]{Kl}, and the proof of the second one may be found in \cite[Lemma 6.5]{BMarXiv}.

\begin{proof}[Proof of Proposition \ref{prop:convergence}]
The required convergence follows from an easy generalization of \cite[Lemma 2.27]{Bri}, where $n=1$, $l$ arbitrary. For this notice that the coset representatives for $Z_l\b{\Gamma}_{\infty}\back\b{\Gamma}$ are $(\l,0,0)\gamma$, where $\l\in M_{l,n}(\Z)$ and $\gamma$ runs through the coset representatives for $\Gamma_{\infty}\back\Gamma$, and use the identity \eqref{eq:cool-id}.

The modularity property follows from the definition of $P_{t,r}$ and absolute convergence.

In order to prove cuspidality we generalise to the case of Jacobi forms the approach of \cite[section III.6]{Kl}, where the case of Siegel Poincar\'e series is considered. We first note that  it suffices to show that the function 
$$(\det y)^{k/2}\exp (-2\pi\tr(S[v]y^{-1}))|P_{t,r}(\tau ,w)|$$
is bounded on $\H_{n,l}$ (c.f. \cite[Lemma 1.3]{Mu89}; note that this condition is independent of the level of $P_{t,r}$ because of strong approximation theorem). 

First we apply Poisson summation formula to the function
$$\varphi(t,r):=\begin{cases}\det\(\frac{4t}{\lag}-S^{-1}[r]\)^{k-(l+n+1)/2}e\(\frac{1}{\lag}\tr(t\tau)\) e(\tr(\T{r}w)) , & t>0, \frac{4t}{\lag}-S^{-1}[r]>0\\
0, & \mbox{otherwise} 
\end{cases}$$
defined on $Sym_n(\R)\times M_{l,n}(\R)$ and for the lattices 
$$\Lambda :=L \times M_{l,n}(\Z)\quad\mbox{and}\quad
\Lambda ':=Sym_n(\Z)\times M_{l,n}(\Z).$$ 
For $2k-l-n>-1$ we obtain
\begin{align*}
\sum_{(t,r)\in\Lambda}\varphi(t,r)
&=\sum_{(b,\mu)\in\Lambda '}\int_{t>0}e\(\frac{1}{\lag}\tr(t(\tau +\lag b))\) dt\\
&\hspace{2cm}\int_{\frac{4t}{\lag}-S^{-1}[r]>0} \det\(\frac{4t}{\lag}-S^{-1}[r]\)^{k-(l+n+1)/2} e(\tr(\T{r}(w+\mu)))dr \\
&=\lag^{n(n+1)/2}\sum_{(b,\mu)\in\Lambda '}\int_{h>0}e\(\tr(h(\tau +\lag b))\)\det (4h)^{k-(l+n+1)/2} dh\\
&\hspace{2cm}\cdot\int_{M_{l,n}(\R)} e(\tr((4S)^{-1}[r](\tau +\lag b))) e(\tr(\T{r}(w+\mu)))dr\\
&=\lag^{n(n+1)/2}2^{nk-n(l+n+1)}\Gamma_n\( k-\frac{l}{2}\)i^{nk}\det(4S)^{n/2}\\
&\hspace{0.5cm}\cdot\sum_{(b,\mu)\in\Lambda '}\det(\tau +\lag b)^{-k}e(-\tr(S[w+\mu](\tau +\lag b)^{-1}))
\end{align*}
Hence
\begin{multline*}
\hspace{-0.4cm}\sum_{\substack{t\in L\\t>0}}\sum_{4t-\lag S^{-1}[r]>0}\det\(4t-\lag S^{-1}[r]\)^{k-(l+n+1)/2}e\(\frac{1}{\lag}\tr(t(\tau_1+\tau_2))\) e(\tr(\T{r}(w_1-w_2)))\\
=\tilde{C}_{\Gamma ,k,n,S}\sum_{\b{g}\in\b{\Gamma}_{\infty}}\det(\tau_1 +\tau_2)^{-k}e(-\tr(S[w_1-w_2](\tau_1 +\tau_2)^{-1}))|^{(1)}_{k,S}\,\b{g},
\end{multline*}
where $\tilde{C}_{\Gamma ,k,n,S}:=\lag^{n(k-l/2)}2^{nk-n(l+n+1)}\Gamma_n\( k-\frac{l}{2}\)i^{nk}\det(4S)^{n/2}$, and $|^{(1)}_{k,S}$ denotes the action with respect to the first variable, namely $(\tau_1,w_1)$.

Now let 
\begin{equation}\label{eq:Pks}
P_{k,S}(z_1,z_2):=\sum_{\b{g}\in Z_l\back\b{\Gamma}}\overline{\chi(\b{g})}\det(\tau_1 +\tau_2)^{-k}e(-\tr(S[w_1-w_2](\tau_1 +\tau_2)^{-1}))|^{(1)}_{k,S}\,\b{g}.
\end{equation}
As we will shortly show, for a fixed $z_2\in\H_{n,l}$ the series $P_{k,S}(\cdot\, ,z_2)$ is absolutely convergent on $\H_{n,l}$ for $k>2n+l$ (and by symmetry also in other variable when  $z_1$ is fixed); then 
\begin{align*}
P_{k,S}(z_1,z_2)=(\tilde{C}_{\Gamma ,k,n,S})^{-1}\sum_{\substack{t\in L\\t>0}}\sum_{4t-\lag S^{-1}[r]>0}&\det\(4t-\lag S^{-1}[r]\)^{k-(l+n+1)/2}\\ 
& \cdot P_{t,r}(\tau_1,w_1)e\(\frac{1}{\lag}\tr(t\tau_2)\) e(-\tr(\T{r}w_2)).
\end{align*}
We will establish the cuspidality of the Poincar\'e series $P_{t,r}(z_1)$ by showing that $P_{k,S}(z_1,z_2)$ is cuspidal in $z_1$ for every fixed $z_2$. We remark here that various properties of $P_{k,S}$ have already been studied in \cite{Ar94} and \cite[Proposition 2]{ZWL}. The rest of the proof, which is implied but not written in \cite{Ar94}, is based on the approach taken by Klingen in \cite{Kl} in the case of Siegel modular forms and offers a slightly different way of proving the cuspidality of $P_{k,S}$ than \cite{ZWL}; for this reason we have decided to include it.

Similarly to the proof of \cite[Proposition 2 and its Corollary]{Kl}, the rest of the proof is based on the comparison between the above series and the function
\begin{equation*}\begin{split}
G(z_1,\hat{z}_2):=\int_{\b{\Gamma}\back\H_{n,l}}|\det(\tau_1 +\hat{\tau}_2)^{k}\det(\tau_1 -\bar{\tau}_2)^{-k}e(\tr(S[w_1-\hat{w}_2](\tau_1 +\hat{\tau}_2)^{-1}))|\\
\cdot |e(-\tr(S[w_1-\bar{w}_2](\tau_1 -\bar{\tau}_2)^{-1}))|(\Delta_{S,k}(z_2))^{1/2}dz_2 
\end{split}\end{equation*}
defined on $\H_{n,l}\times\H_{n,l}$; 
if $k>2n+l$, then by \cite[Lemma 3 and Lemma 1.(2)]{ZWL} the defining integral is finite for any $z_1,\hat{z}_2\in\H_{n,l}$. Moreover, similarly to \cite[page 80]{Kl}, for any compact subset $\mathcal{K} \subset \mathcal{H}_{n,l}$ there exists a constant $c(\mathcal{K})>0$ such that 
\begin{equation} \label{eq:bound}
G(z_1,\hat{z}_2)\geq c(\mathcal{K})\qquad\mbox{ for}\quad (z_1,\hat{z}_2 ) \in \mathcal{H}_{n,l}\times \mathcal{K}. 
\end{equation}
In order to proceed we recall from \cite[p. 192-193]{Ar94} that the map
\begin{equation*}
\b{g}=(\l,\mu,\kappa)\begin{pmatrix} a & b\\ c & d\end{pmatrix}\quad\longmapsto\quad \tilde{\b{g}}:=(-\l,\mu,-\kappa)\begin{pmatrix} a & -b\\ -c & d\end{pmatrix}
\end{equation*}
defines an involutive automorphism of $\b{G}^{n,l}(\R)$ with the following properties:
$$\tilde{\b{g}}(-\bar{\tau},\bar{w})=(-\overline{g\tau},\overline{\b{g}w}),\qquad\qquad \overline{J_{k,S}(\tilde{\b{g}},(-\bar{\tau},\bar{w}))}=J_{k,S}(\b{g}, (\tau ,w))$$
and (c.f. \cite[Lemma 1.4]{Ar94})
\begin{align*}
&\det(g\tau_1 -\bar{\tau}_2)^{-k}e(-\tr(S[\b{g}w_1-\bar{w}_2](g\tau_1 -\bar{\tau}_2^{-1})))J_{k,S}(\b{g},(\tau_1,w_1))^{-1}=\\
&\det(\tau_1 +\tilde{g}^{-1}(-\bar{\tau}_2))^{-k}e(-\tr(S[w_1-\tilde{\b{g}}^{-1}\bar{w}_2](\tau_1 +\tilde{g}^{-1}(-\bar{\tau}_2)^{-1})))J_{k,S}(\tilde{\b{g}}^{-1},(-\bar{\tau}_2,\bar{w}_2))^{-1} ,
\end{align*}
where $\b{g}w_1:=w_1 \lambda(g,\tau_1)^{-1}+\l g \tau_1+\mu$.

Hence, for $\b{g}=(\l,\mu,0)g\in\b{\Gamma}$ we have
\begin{align*}
G&(\b{g}z_1,\hat{z}_2)|\det(g\tau_1 +\hat{\tau}_2)^{-k}e(-\tr(S[\b{g}w_1-\hat{w}_2](g\tau_1 +\hat{\tau}_2)^{-1}))J_{k,S}(\b{g},z_1)^{-1}|\\
&=\int_{\b{\Gamma}\back\H_{n,l}} |\det(g\tau_1 -\bar{\tau}_2)^{-k}e(-\tr(S[\b{g}w_1-\bar{w}_2](g\tau_1 -\bar{\tau}_2)^{-1}))J_{k,S}(\b{g},z_1)^{-1}|\\
&\hspace{1.5cm}\cdot (\Delta_{S,k}(z_2))^{1/2}dz_2 \\
&=\int_{\b{\Gamma}\back\H_{n,l}} |\det(\tau_1 -\overline{g^{-1}\tau_2})^{-k}e(-\tr(S[w_1-\overline{\b{g}^{-1}w_2}](\tau_1 -\overline{g^{-1}\tau_2})^{-1}))|\\
&\hspace{1.5cm}\cdot |\overline{J_{k,S}(\b{g}^{-1},(\tau_2,w_2))}|^{-1}(|J_{k,S}(\b{g}^{-1},z_2)|^2\Delta_{S,k}(\b{g}^{-1}z_2))^{1/2}dz_2\\
&=\int_{\b{\Gamma}\back\H_{n,l}} |\det(\tau_1 -\overline{g^{-1}\tau_2})^{-k}e(-\tr(S[w_1-\overline{\b{g}^{-1}w_2}](\tau_1 -\overline{g^{-1}\tau_2})^{-1}))|(\Delta_{S,k}(\b{g}^{-1}z_2))^{1/2}dz_2.
\end{align*} 

We now fix a $\hat{z}_2$ and write $\mathcal{C}$ for the constant $c(\mathcal{K})$ appearing in equation (\ref{eq:bound}) above, for some compact set $\mathcal{K} \ni \hat{z}_2$. By triangle inequality, 
\begin{align*}
\mathcal{C} &\cdot |P_{k,S}(z_1,\hat{z}_2)|\\
& \leq \sum_{\b{g} \in Z_l\back\b{\Gamma}} G(\b{g}z_1,\hat{z}_2)|\det(g\tau_1 +\hat{\tau}_2)^{-k}e(-\tr(S[\b{g}w_1-\hat{w}_2](g\tau_1 +\hat{\tau}_2)^{-1}))J_{k,S}(\b{g},z_1)^{-1}|\\
&= \sum_{\b{g} \in Z_l\back\b{\Gamma}} \int_{\b{\Gamma}\back\H_{n,l}} |\det(\tau_1 -\overline{g^{-1}\tau_2})^{-k}e(-\tr(S[w_1-\overline{\b{g}^{-1}w_2}](\tau_1 -\overline{g^{-1}\tau_2})^{-1}))|\\
&\hspace{3cm}\cdot(\Delta_{S,k}(\b{g}^{-1}z_2))^{1/2}dz_2\\
&\leq \int_{\H_{n,l}} |\det(\tau_1 -\bar{\tau}_2)^{-k}e(-\tr(S[w_1-\bar{w}_2](\tau_1 -\bar{\tau}_2)^{-1}))|(\Delta_{S,k}(z_2))^{1/2}dz_2<\infty .
\end{align*}

The finiteness follows again from \cite[Lemma 3]{ZWL}, and thus proves absolute convergence of $P_{k,S}(z_1,\hat{z}_2)$ (because what we really bound here is the series of absolute values of the consecutive terms of $P_{k,S}(z_1,\hat{z}_2)$). 

In fact, a closer look at the proof of \cite[Lemma 3]{ZWL} reveals that if $\b{\xi}\in\b{G}^{n,l}(\R)$ is such that $\b{\xi}(i1_n,0)=(\tau_1,w_1)$ and $c_{k,S}>0$ is a constant as in \cite[p. 716]{ZWL}, then
\begin{multline*}
c_{k,S}^{-1}2^{-nk}\int_{\H_{n,l}} |\det(\tau_1 -\bar{\tau}_2)^{-k}e(-\tr(S[w_1-\bar{w}_2](\tau_1 -\bar{\tau}_2)^{-1}))|(\Delta_{S,k}(z_2))^{1/2}dz_2\\
=|J_{k,S}(\b{\xi}, (i1_n,0))|\int_{\H_{n,l}} |\det(i1_n -\bar{\tau}_2)^{-k}e(-\tr(S[\bar{w}_2](i1_n -\bar{\tau}_2)^{-1}))|(\Delta_{S,k}(z_2))^{1/2}dz_2,
\end{multline*}
where 
$$J_{k,S}(\b{\xi}, (i1_n,0))=\Delta_{S,k}((i1_n,0))^{1/2}\Delta_{S,k}(z_1)^{-1/2}=(\det y_1)^{-k/2}\exp (2\pi\tr(S[v_1]y_1^{-1}))$$
and the integral, now independent of $z_1$, is explicitly computed in the remaining part of the proof and turns out to be finite if $k>2n+l$. This proves that for every fixed $\hat{z}_2$ the function 
\[
(\det y_1)^{k/2}\exp (-2\pi\tr(S[v_1]y_1^{-1}))|P_{k,S}(z_1,\hat{z}_2)| 
\]
is bounded on $\mathcal{H}_{n,l}$, and hence $P_{k,S}(z_1,\hat{z}_2)$ is a cusp form in the $z_1$ variable. From the discussion above on the relation between $P_{k,S}$ and $P_{t,r}$ it follows that the Poincar\'e series $P_{t,r}$ is a cusp form. 
\end{proof}

\begin{cor}\label{cor:kernel}
For $k>2n+l$, for every $z_2\in\H_{n,l}$, the function 
\begin{align*}
K(z_1,z_2):=C_{\b{\Gamma}, k,n,S}^{-1}\sum_{\substack{t\in L\\t>0}}\sum_{4t-\lag S^{-1}[r]>0}&\det(4t-\lag S^{-1}[r])^{k-(n+l+1)/2} P_{t,r}(\tau_1,w_1)\\
&\cdot e\(-\frac{1}{\lag}\tr(t\bar\tau_2)\) e(-\tr(\T{r}\bar w_2))\\
&\hspace{-4.5cm}=C_{\b{\Gamma}, k,n,S}^{-1} \tilde{C}_{\Gamma ,k,n,S} P_{k,S}((\tau_1 ,w_1),(-\bar{\tau}_2,\bar{w}_2)),
\end{align*}
where $P_{k,S}$ is defined in \eqref{eq:Pks}, is absolutely and uniformly convergent in $z_1$ on compact subsets of $\H_{n,l}$ and defines a cusp form in $S_{k,S}^n(\b{\Gamma},\chi)$.
The use of constant 
\begin{equation}\label{eq:kernel constant}
C_{\b{\Gamma}, k,n,S}:=vol(\b{\Gamma}\back\H_{n,l})^{-1}\det(2S)^{-n/2}\Gamma_n\(k-\frac{n+l+1}{2}\)\(\frac{\pi}{\lag}\)^{n(n+l+1)/2-nk}
\end{equation}
will be justified in Theorem \ref{prop:projection}.(a). Moreover, for any $f\in S^n_{k,S}(\b{\Gamma},\chi)$:
$$<f,K(\,\cdot\, ,z_2)> =f(z_2).$$
\end{cor}
\begin{proof}
The first statement is a direct consequence of the proof of Proposition \ref{prop:convergence}. In particular, $K(z_1,z_2)$ can be given by two formulas indicated above. The second equality implies that, up to a constant and appearance of $\chi$, it coincides with the function studied in \cite[Proposition 2]{ZWL}. From this it is not difficult to see that the formula for the inner product remains valid in our setting. In fact, this formula naturally comes out in the proof of Theorem \ref{prop:projection} below, where we carry out more general computations.
\end{proof}

\begin{rem}
It seems to us that for convergence reasons the series $K(z_1,z_2)$, when written using the formula \eqref{eq:Pks} for $P_{k,S}$, should be defined as a sum which is taken only over $\b{g}\in Z_l\back\b{\Gamma}$, and not over all $\b{g}\in\b{\Gamma}$ as in \cite{ZWL}. Our definition coincides also with the one used by \cite{Ar94}.
\end{rem}

Corollary \ref{cor:kernel} implies that $K(z_1,z_2)$ is the reproducing kernel for $S^n_{k,S}(\b{\Gamma},\chi)$. This means that in fact
$$K(z_1,z_2)=\sum_{i=1}^d f_i(z_1)\overline{f_i(z_2)},$$
where $\{ f_1,\ldots ,f_d\}$ is an orthonormal basis for $S_{k,S}^n(\b{\Gamma},\chi)$.
As a consequence,
$$<f(z_1),K(z_1,z_2)>=\sum_{i=1}^d <f(z_1),f_i(z_1)>f(z_2)\in S^n_{k,S}(\b{\Gamma},\chi)$$
for any $f\in \f{M}^n_{k,S}(\b{\Gamma} ,\chi)$, as long as the sum is finite. The finiteness condition will follow if we assume that $f$ is of bounded growth. We define this notion as follows: 

We say that a function $f : \mathcal{H}_{n,l} \rightarrow \mathbb{C}$ is of bounded growth, and we write $f \in \mathcal{B}^n_{k,S}(\b{\Gamma})$ if the integral
$$\int_{Y_n}\int_{M_{l,n}(\R)}\int_{X_n}\int_{U_n} |f(z)|e^{-2\pi\tr\(\frac{t}{\lag} y\)}e^{-2\pi\tr (\T{r}v)}\Delta_{S,k}(z)dz$$
is finite for every pair $(t,r) \in \frac{1}{2}Sym_n(\mathbb{Z}) \times M_{l,n}(\mathbb{Z})$, with $t >0$ and $4t -\lag S^{-1}[r] > 0$. Here
$$X_n:=\{ x\in M_n(\R):\T{x}=x, |x_{i,j}|\leq \lambda_{\Gamma}/2 \mbox{ for all } i,j\} ,$$
$$U_n:=\{ u\in M_{l,n}(\R): |u_{i,j}|\leq 1/2 \mbox{ for all } i,j\} ;$$
if it does not lead to confusion, we will omit the subscripts $n$ in $X_n, U_n$ and $Y_n$.
We recall here that $\Delta_{S,k}(z)= \det(y)^k e^{-4 \pi \tr(S[v] y^{-1})}$ and $dz=(\det y)^{-(l+n+1)}dxdydudv$, where $\tau =x+iy$ and $w=u+iv$, and we remark that
$$\{(x+iy,u+iv): x\in X, y\in Y, u\in U, v\in M_{l,n}(\R)\} =\bigcup_{\b{g}\in\b{\Gamma}_{\infty}\back\b{\Gamma}} \b{g}\mathcal{F}_{\b{\Gamma}},$$
where $\mathcal{F}_{\b{\Gamma}}$ is a fundamental domain for $\b{\Gamma}$.

We will be writing $\mathcal{B}^n_{k,S}$ for the union of $\mathcal{B}^n_{k,S}(\b{\Gamma})$ over all congruence subgroups $\b{\Gamma}$. Our definition should be seen as a generalisation of the notion of bounded growth introduced by Sturm in \cite{St} in the case of Siegel modular forms. \newline

We note that for $k > 2n + \ell$, we have $S_{k,S}^n(\b{\Gamma},\chi) \subset \mathcal{B}^n_{k,S}(\b{\Gamma})$. Indeed, if $f \in S_{k,S}^n(\b{\Gamma},\chi)$, then the function $\Delta_{S,k}(z)^{1/2}|f(z)|$ is bounded on $\mathcal{H}_{n,l}$ (see \cite[Lemma 2.6]{Z89}) and hence there exists a constant $M > 0$ such that 
\begin{multline*}
\int_{Y_n}\int_{M_{l,n}(\R)}\int_{X_n}\int_{U_n} |f(z)|e^{-2\pi\tr\(\frac{t}{\lag} y\)}e^{-2\pi\tr (\T{r}v)}\Delta_{S,k}(z)dz \leq \\
M \int_{Y_n}\int_{M_{l,n}(\R)}\int_{X_n}\int_{U_n} e^{-2\pi\tr\(\frac{t}{\lag} y\)}e^{-2\pi\tr (\T{r}v)} \det(y)^{k/2} e^{-2 \pi \tr(S[v] y^{-1})}  dz.
\end{multline*}
The last integral is equal to
\begin{multline*}
vol(X_n \times U_n) \int_{Y_n} \det(y)^{k/2- (l+n+1)}  e^{-2\pi\tr\(\frac{t}{\lag} y\)} \left( \int_{M_{l,n}(\R)} e^{-2\pi\tr (S[v] y^{-1}+\T{r}v)} dv \right) dy\\
\stackrel{\eqref{eq:cool-id}}{=}\det(2S)^{-n/2}vol(X_n \times U_n) \int_{Y_n} \det(y)^{\frac{k-l}{2}- n}  e^{-\frac{\pi}{2 \lag}\tr\( (4t - \lag S^{-1}[r])y \)} d^{\times}y,
\end{multline*}
where $d^{\times}y = \det(y)^{-1}dy$; and is finite since $(4t - \lag S^{-1}[r])y > 0$ and $\frac{k-l}{2}- n > 0$. \newline

We are now ready to state the main result of this section. The following Theorem (except the statement (b)) is a generalization of \cite[Proposition 1 and Theorem 1]{St} to the Jacobi setting. \newline

\begin{thm}\label{prop:projection}
Let $k>2n+l$ and $f\in \f{M}^n_{k,S}(\b{\Gamma} ,\chi)\cap \mathcal{B}^n_{k,S}(\b{\Gamma})$ with the Fourier expansion \eqref{eq:Fexp-in-y}. For $t>0$ and $r\in M_{l,n}(\Z)$ such that $4t-\lambda_\Gamma S^{-1}[r]>0$ we define the coefficient
\begin{align*}
c(t,r):= &\Gamma_n\(k-\frac{n+l+1}{2}\)^{-1}\(\frac{\pi}{\lag}\)^{-n(n+l+1)/2+nk}\det(4t-\lag S^{-1}[r])^{k-(l+n+1)/2}\\
&\cdot\int_{Y_n} A_{t,r}(y)e^{-\pi\tr((\frac{2t}{\lag}- S^{-1}[r])y)}\det (y)^{k-l/2-n-1}dy .
\end{align*}
Then 
\begin{align*}
Hol(f)(z_2):&=<f(z_1),K(z_1,z_2)>\\
&=\sum_{t>0}\sum_{4t-\lambda_\Gamma S^{-1}[r]>0}c(t,r)e\(\frac{1}{\lag}\tr(t\tau_2)\) e(\tr(\T{r}w_2))\in S^n_{k,S}(\b{\Gamma},\chi)
\end{align*}
and 
$$<f,g>=<Hol(f), g>\quad\mbox{ for all }\quad g \in S^n_{k,S}(\b{\Gamma},\chi).$$
Moreover,
\begin{enumerate}
\item[(a)] if $f(\tau,w)=\sum_{t,r}\tilde{c}(t,r)e\(\frac{1}{\lag}\tr(t\tau_2)\) e(\tr(\T{r}w_2))\in S^n_{k,S}(\b{\Gamma},\chi)$, then 
$$<f,P_{t,r}>=C_{\b{\Gamma}, k,n,S}\det(4t-\lag S^{-1}[r])^{-k+(n+l+1)/2}\,\tilde{c}(t,r),$$
where the constant $C_{\b{\Gamma}, k,n,S}$ is as in \eqref{eq:kernel constant}; hence $Hol(f)=f$; 
\item[(b)] if $f(\tau,w)=\sum_{t,r}p_{t,r}(y^{-1})e\(\frac{1}{\lag}\tr(t\tau)\) e(\tr(\T{r}w))\in N_{k,S}^{n,D}(\b{\Gamma},\chi)$ and $k>n+\frac{l}{2}+D$, then
\begin{align*}
<f,P_{t,r}>=&vol(\mathcal{F}_{\b{\Gamma}})^{-1}\det(2S)^{-n/2}\(\frac{\pi}{\lag}\)^{n(n+l+1)/2-nk}\Gamma_n\(k-\frac{n+l+1}{2}-D\)\\
&\cdot R_{t,r}\(-\frac{\partial}{\partial y}\)\left[\det(y)^{-k+D+(n+l+1)/2}\right]_{y=h}
\end{align*}
and so
\begin{align*}
Hol(f)(z_2)=&\Gamma_n\(k-\frac{n+l+1}{2}\)^{-1}\Gamma_n\(k-\frac{n+l+1}{2}-D\)\\
&\cdot\sum_{t,r}e\(\frac{1}{\lag}\tr(t\tau_2)\) e(\tr(\T{r}w_2))\det(4t-\lag S^{-1}[r])^{k-(n+l+1)/2}\\
&\hspace{1cm} \cdot R_{t,r}\(-\frac{\partial}{\partial y}\)\left[\det(y)^{-k+D+(n+l+1)/2}\right]_{y=h} ,
\end{align*}
where $R_{t,r}(y):=(\det y)^{D}p_{t,r}(y^{-1})$ is a polynomial in the entries of $y$, $h:=4t-\lag S^{-1}[r]>0$, $\frac{\partial}{\partial y}:=[\frac12 (1+\delta_{ij}\frac{\partial}{\partial y_{ij}})]_{1\leq i\leq j\leq n}$ and $\delta_{ij}$ denotes Dirac delta.
\end{enumerate}
\end{thm}
\begin{proof}
First we compute the inner product of $f$ and $P_{t,r}$:
\begin{align*}
vol(\mathcal{F}_{\b{\Gamma}})&<f,P_{t,r}>\\
&=\int_{\mathcal{F}_{\b{\Gamma}}}\sum_{\b{g}\in\b{\Gamma}_{\infty}\back\b{\Gamma}}\chi(\b{g})f(z)e(-\tr(\T{r}\overline{\b{g}w}+\frac{t}{\lag}\overline{\b{g}\tau}))\overline{J_{k,S}(\b{g},z)^{-1}}\Delta_{S,k}(z)dz\\
&=\int_{\mathcal{F}_{\b{\Gamma}}}\sum_{\b{g}\in\b{\Gamma}_{\infty}\back\b{\Gamma}}\chi(\b{g})J_{k,S}(\b{g},z)f(z)e(-\tr(\T{r}\overline{\b{g}w}+\frac{t}{\lag}\overline{\b{g}\tau}))|J_{k,S}(\b{g},z)|^{-2}\Delta_{S,k}(z)dz\\
&=\sum_{\b{g}\in\b{\Gamma}_{\infty}\back\b{\Gamma}}\int_{\mathcal{F}_{\b{\Gamma}}} f(\b{g}z)e(-\tr(\T{r}\overline{\b{g}w}+\frac{t}{\lag}\overline{\b{g}\tau}))\Delta_{S,k}(\b{g}z)dz\\
&=\int_Y \int_{M_{l,n}(\R)}\int_X\int_U f(\tau ,w)e(-\tr(\T{r}\bar{w}+\frac{t}{\lag}\bar\tau))\Delta_{S,k}(z)dz <\infty ,
\end{align*}
where the interchange of integration and summation is justified by the bounded growth assumption. Further, since
\begin{align*}
\int_X\int_U &f(\tau ,w)e(-\tr(\T{r}\bar{w}+t\bar\tau/\lag))dudx\\
&=\sum_{t',r'} A_{t',r'}(y)\int_X e(\tr(t'x-t(x-iy))/\lag)dx\int_U e(\tr(\T{r'}(u+iv)-\T{r}(u-iv)))du\\
&=A_{t,r}(y)e^{-2\pi\tr(ty)/\lag}e^{-4\pi\tr(\T{r}v)}
\end{align*}
and
\begin{align*}
\int_Y\int_{M_{l,n}(\R)} &A_{t,r}(y)e^{-2\pi\tr(ty)/\lag}e^{-4\pi\tr(\T{r}v)}\Delta_{S,k}(z)(\det y)^{-(l+n+1)}dvdy\\
&=\int_Y A_{t,r}(y)e^{-2\pi\tr(ty)/\lag}(\det y)^{k-(l+n+1)}dy \int_{M_{l,n}(\R)} e^{4 \pi \tr(-S[v] y^{-1}-y\T{r}vy^{-1})}dv\\
&=\det(2S)^{-n/2}\int_Y A_{t,r}(y)e^{-2\pi\tr(ty)/\lag}e^{\pi\tr(S^{-1}[r]y)}(\det y)^{k-(l/2+n+1)}dy ,
\end{align*}
we obtain a finite quantity
\[
vol(\mathcal{F}_{\b{\Gamma}})<f,P_{t,r}>=\det(2S)^{-n/2}\int_Y A_{t,r}(y)e^{-2\pi\tr(ty)/\lag}e^{\pi\tr(S^{-1}[r]y)}(\det y)^{k-(l/2+n+1)}dy.
\]

If $f\in S^n_{k,S}(\b{\Gamma},\chi)$ with Fourier coefficients $\tilde{c}(t,r)$, then $A_{t,r}(y)=\tilde{c}(t,r)e^{-2\pi\tr(ty)/\lag}$. As we have seen above, $f$ is of bounded growth and by applying identity \eqref{eq:int-det} we obtain 
\begin{multline*}
vol(\mathcal{F}_{\b{\Gamma}})<f,P_{t,r}>=\det(2S)^{-n/2}\int_Y \tilde{c}(t,r)e^{-\pi\tr((\frac{4t}{\lag}- S^{-1}[r])y)}(\det y)^{k-(l/2+n+1)}dy\\
=\det(2S)^{-n/2}\Gamma_n\(k-\frac{n+l+1}{2}\)\pi^{n(n+l+1)/2-nk}\det\(\frac{4t}{\lag}-S^{-1}[r]\)^{-k+(n+l+1)/2}\tilde{c}(t,r) .
\end{multline*}
Hence $Hol(f)=f$, which provides another proof of the second statement in Corollary \ref{cor:kernel}.

Now for any $f\in \f{M}^n_{k,S}(\b{\Gamma} ,\chi)\cap \mathcal{B}^n_{k,S}(\b{\Gamma})$, since $K(z_1,z_2)$ is the reproducing kernel for $S^n_{k,S}(\b{\Gamma},\chi)$, we have
\begin{align*}
< &f(z_1),K(z_1,z_2)>\\
&=C_{\b{\Gamma}, k,n,S}^{-1}\sum_{t>0}\sum_{4t-\lag S^{-1}[r]>0}\det(4t-\lag S^{-1}[r])^{k-(n+l+1)/2} e\(\frac{1}{\lag}\tr(t\tau_2)\) e(\tr(\T{r}w_2))\\
&\hspace{5cm}\cdot <f(z_1), P_{t,r}(z_1)>\\
&=\sum_{t>0}\sum_{4t-\lambda_\Gamma S^{-1}[r]>0}c(t,r)e\(\frac{1}{\lag}\tr(t\tau_2)\) e(\tr(\T{r}w_2))\in S^n_{k,S}(\b{\Gamma},\chi) .
\end{align*}
Furthermore, for any $g\in S^n_{k,S}(\b{\Gamma},\chi)$, 
\begin{align*}
<&Hol(f),g>=\sum_i<f(z_1),f_i(z_1)><f_i(z_2), \sum_j <g(z_1),f_j(z_1)>f_j(z_2)>\\
&=\sum_i <f,f_i>\overline{<g(z_1),f_i(z_1)>}=<f(z_2),\sum_i <g,f_i>f_i(z_2)>=<f(z_2),g(z_2)>.
\end{align*}

If $f\in N_{k,S}^{n,D}(\b{\Gamma},\chi)$ is a nearly holomorphic Jacobi form, then we can also explicitly evaluate the integral over $Y$ (c.f. \cite[eq. (2.165), p. 89]{CP}). Now $A_{t,r}(y)=p_{t,r}(y^{-1})e^{-2\pi\tr(ty)/\lag}$, where $p_{t,r}(y^{-1})$ is a polynomial of total degree $D$ in the entries of $y^{-1}$. Since $y^{-1}=(\det y)^{-1}\T{Adj(y)}$ with $Adj(y)$ denoting the adjoint of $y$, we can replace each homogeneous polynomial of degree $d$ occurring in $p_{t,r}(y^{-1})$ by $(\det y)^{-d} R^d_{t,r}(y)$, where $R^d_{t,r}(y)$ is a polynomial in $y_{ij}$, $1\leq i\leq j\leq n$, of homogeneous degree $d(n-1)$. We adjust it further to $R_{t,r}(y):=\sum_{d=0}^D (\det y)^{D-d} R^d_{t,r}(y)=(\det y)^{D}p_{t,r}(y^{-1})$. Then, as long as $k>n+\frac{l}{2}+D$, 
\begin{multline*}
\det(2S)^{n/2}vol(\mathcal{F}_{\b{\Gamma}})<f,P_{t,r}>=\int_Y R_{t,r}(y)e^{-\pi\tr((\frac{4t}{\lag}-S^{-1}[r])y)}(\det y)^{k-(l/2+n+1)-D}dy\\
=\(\frac{\pi}{\lag}\)^{n(n+l+1)/2-nk}\Gamma_n\(k-\frac{n+l+1}{2}-D\)R_{t,r}\(-\frac{\partial}{\partial y}\)\left[\det(y)^{-k+D+(n+l+1)/2}\right]_{y=h},
\end{multline*}
where $h:=4t-\lag S^{-1}[r]>0$, $\frac{\partial}{\partial y}:=[\frac12 (1+\delta_{ij}\frac{\partial}{\partial y_{ij}})]_{1\leq i\leq j\leq n}$ and $\delta_{ij}$ denotes Dirac delta.
\end{proof}

We can now derive the following important corollary from the calculation done above.

\begin{cor} \label{Hol preserves algebraicity} Let $f \in N_{k,S}^{n,D}(\b{\Gamma},\chi,\overline{\mathbb{Q}}) \cap \mathcal{B}^n_{k,S}(\b{\Gamma})$, with $k > \max\{n + \frac{l}{2} + D, 2n+ l\}$. Then $Hol(f) \in S_{k,S}^n(\b{\Gamma},\chi,\overline{\mathbb{Q}})$.
\end{cor}
\begin{proof} This follows immediately from the Theorem above. Indeed, if we write 
\[
f(\tau,w)=\sum_{t,r}p_{t,r}(y^{-1})e\(\frac{1}{\lag}\tr(t\tau)\) e(\tr(\T{r}w)),
\]
then the polynomials $p_{t,r}(y^{-1})$ have algebraic coefficients. Since $k > n+ \frac{l}{2}+D$,
\[
\frac{\Gamma_n\(k-\frac{n+l+1}{2}-D\)}{\Gamma_n\(k-\frac{n+l+1}{2}\)} \in \mathbb{Q}^{\times},
\]
because $\Gamma(m+\frac{1}{2})$ is a rational multiple of $\sqrt{\pi}$ for any $m \in \mathbb{N}$. Then by Theorem \ref{prop:projection} (b) above, $Hol(f)$ has Fourier expansion with coefficients in $\overline{\mathbb{Q}}$ since (in the notation of the Theorem) $R_{t,r}\(-\frac{\partial}{\partial y}\)\left[\det(y)^{-k+D+(n+l+1)/2}\right]_{y=h} \in \overline{\mathbb{Q}}$, as $h=4t-\lag S^{-1}[r] > 0$ is a symmetric positive definite matrix with rational  entries for all $(t,r)$.
\end{proof}

We conclude this section with another Corollary of Theorem \ref{prop:projection}, which we will use in the next section.
Namely, if the nearly holomorphic Siegel-Jacobi form has a particular kind of Fourier expansion, then the required bound on the weight $k$ can be improved.

\begin{cor} \label{improved bound projection} Let $f \in N_{k,S}^{n,D}(\b{\Gamma},\chi,\overline{\mathbb{Q}}) \cap \mathcal{B}^n_{k,S}(\b{\Gamma})$ and assume that $D = m \, n$ for some $m \in \mathbb{Z}_+$. Assume further that 
\[
f(\tau,w)=\det(y)^{-m} \sum_{t,r} Q_{t,r}(y)e\(\frac{1}{\lag}\tr(t\tau)\) e(\tr(\T{r}w)),
\]
for some polynomials $Q_{t,r}(y) \in \overline{\mathbb{Q}}[y_{ij}]$, that is, the polynomials in the entries of $y$, with algebraic coefficients. Then $Hol(f) \in S^n_{k,S}(\b{\Gamma},\chi, \overline{\mathbb{Q}})$ if $k > \max\{n + \frac{l}{2} + m,2n+l\}$.
\end{cor}

\begin{proof} This follows from the proof of Theorem \ref{prop:projection} (b) after observing that in the notation used there, $A_{t,r}(y) = \det(y)^{-m} Q_{t,r}(y)e^{-2\pi\tr(ty)/\lag}$. In particular the relevant integrals are of the form
\begin{multline*}
\int_{Y} A_{t,r}(y)e^{-\pi\tr((\frac{2t}{\lag}-S^{-1}[r])y)}\det (y)^{k-l/2-n-1}dy = \\
\int_{Y} Q_{t,r}(y)e^{-\pi\tr((\frac{4t}{\lag}-S^{-1}[r])y)}\det (y)^{k-l/2-n-1-m}dy,
\end{multline*}
and thus we can evaluate it using the formula in \cite[eq. (2.165), p. 89]{CP} directly, without the need to apply the transformation involving the adjoint of $y$; that is, it suffices to assume that $k > n + \frac{l}{2}+m$ to obtain that it is equal to
\[
\(\frac{\pi}{\lag}\)^{n(n+l+1)/2-nk}\Gamma_n\(k-\frac{n+l+1}{2}-m\)Q_{t,r}\(-\frac{\partial}{\partial y}\)\left[\det(y)^{-k+m+(n+l+1)/2}\right]_{y=h},
\]
where $h:=4t-\lag S^{-1}[r]>0$. The corollary now follows by observing that we still have
\[
\frac{\Gamma_n\(k-\frac{n+l+1}{2}-m\)}{\Gamma_n\(k-\frac{n+l+1}{2}\)} \in \mathbb{Q}^{\times},
\]

which concludes the proof. 
\end{proof}

\begin{rem}
The construction of holomorphic projection by Sturm \cite{St}, which we extended to the Siegel-Jacobi setting in Theorem \ref{prop:projection}, was generalized by Courtieu and Panchishkin in \cite[Theorem 2.16]{CP}. Under a weaker growth condition, which they call ``of moderate growth'', they were able to define a projection to holomorphic (not necessarily cuspidal) Siegel modular forms. It seems that a similar extension is possible in the Siegel-Jacobi setting, and one could introduce there a notion of moderate growth. One of the reasons we haven't pursue this here is that the crucial difficulty to derive any applications lies in establishing which functions are actually of the required growth. Therefore we content ourselves with the case which already provides the necessary tool for our purposes. 
\end{rem}

In the next section, Lemma \ref{Bounded growth} provides a sufficient condition for bounded growth of certain $C^{\infty}$ functions, which is then verified in Proposition \ref{Eisenstein growth} in case of nearly holomorphic Siegel-Jacobi Eisenstein series.

\section{Algebraicity results without assuming Property A}

In this last part we use the theory developed in the previous section to obtain algebraicity results without assuming Property A. In particular, the aim is to establish the following theorem, whose proof will be concluded in the end of the section. We note that in comparison to Theorem \ref{Main Theorem on algebraicity}, the bound on $k$ is slightly weaker. This is because we need to ensure that the growth condition is met, in order to be able to apply the holomorphic projection discussed above. As it was mentioned in the introduction, the main tool to establish the theorem below is the doubling identity of Theorem \ref{thm:dm_identity}, however the approach for proving the algebraicity of the ratio of the Petersson inner products is different (compare the proofs of Theorem \ref{ratio of inner products} above with Theorem \ref{ratio of inner products_v2} below). Finally, as we indicated above, we  restrict ourselves to the case of $F = \mathbb{Q}$, but we expect that our results could be extended to totally real fields (see also Remark \ref{extension to totally real} below). 

We fix a fractional ideal $\mathfrak{b}$ and an integral ideal $\mathfrak{c}$ of $\mathbb{Q}$, and write $\b{\Gamma}$ for the congruence subgroup introduced in section 3. 

\begin{thm}\label{Main Theorem on algebraicity v2}
Let $n >1$, and assume that $S$ satisfies property $M_{\f{p}}^+$ for all prime ideals $\mathfrak{p}$ of $\mathbb{Q}$ prime to $\mathfrak{c}$. Let $0 \neq \b{f} \in S_{k,S}^{n}(\b{\Gamma},\overline{\mathbb{Q}})$ be an eigenfunction of all $T(\f{a})$, where $k > 6n +2l +1$,  and $\chi$ a Dirichlet character such that $\chi(-1) =(-1)^k$.
Let $\sigma \in  \mathbb{Z}$ be such that 
\begin{enumerate}
\item[(i)] $k/2 -2n -l/2 > \sigma/2  > n +l/2 +1$, 
\item[(ii)] $ \sigma  - k  \in 2 \mathbb{Z}$,
\end{enumerate}
Under these conditions
\[
\frac{\b{\Lambda}(\sigma/2,\b{f},\chi)}{\pi^{e_{\sigma}} <\b{f},\b{f}>} \in \overline{\mathbb{Q}},
\]
where  
$$e_{\sigma} =  n (k - l + \sigma) - e,\quad e := \begin{cases}  n^2 + n - \sigma + l/2, & \mbox{if }  l  \in 2 \mathbb{Z}, \\ n^2, & \mbox{otherwise}. \end{cases} $$
\end{thm}

\vspace{0.5cm}

First, in a few steps, we will establish the range for $s$ in which the Eisenstein series $E^{2n}|_{k,S}\b{\rho}(z,s,\chi)$ is of bounded growth.

\begin{lem} \label{Bound Jacobi form} Let $g\in M^n_{k,S}(\b{\Gamma}' ,\psi)$ be a Jacobi form of weight $k \in \frac{1}{2}\mathbb{Z}$ for some congruence subgroup $\b{\Gamma}'$ and character $\psi$. Then there exists a positive real constant $c_1$ such that 
\[
|g(z)| = |g(\tau,w)| \leq c_1   \left(\prod_{i=1}^n (1- \lambda_j^{-k}) \right)  e^{2 \pi \tr(y^{-1}S[v])},
\]
where $\lambda_j$ are the eigenvalues of $y$. Here we write $\tau = x + i y$ and $w = u + i v$. 
\end{lem}

\begin{proof}

The proof of this lemma is similar to an analogue statement proved in \cite[Corollary 1, (B)]{St}. Indeed, we define the function $\phi(z) := \det(y)^{k/2} e^{-2 \pi \tr(y^{-1}S[v])} |g(z)|$, where as usual we write $z =(\tau,w)$ with $\tau = x+iy$ and $w = u + i v$. Then $\phi (\gamma z) = \phi(z)$ for all $\gamma \in \b{\Gamma}'$. 
We write $\b{\Gamma}' = H(\Z)\rtimes\Gamma'$ and we choose 
$$\b{\Omega}:=\{ (\tau,\lambda\tau +\mu): \tau\in\mathcal{F}_{\Gamma'} , \lambda ,\mu\in M_{l,n}(\R), \forall_{i,j}\, |\lambda|,|\mu|\leq 1\}$$
as a fundamental domain of $  \b{\Gamma}' \back \H_{n,l}$, where $\mathcal{F}_{\Gamma'}$ is Siegel's fundamental domain for $\Gamma' \back \mathbb{H}_n$. Then there exists a constant $c$ depending only on $g$ such that $|g(z)| \leq c$ for all $z \in \b{\Omega}$, and hence in particular thanks to the invariance of $\phi$ we have
$\phi(\gamma z) \leq c \det(y)^{k/2}   e^{-2 \pi \tr(y^{-1}S[v])}$ for all $z \in \b{\Omega}$ and $\gamma \in \b{\Gamma}'$.

We now consider any $z$. We pick a $\gamma$ such that $\gamma z \in \b{\Omega}$. Then 
\[
\phi(z) = \phi(\gamma^{-1} (\gamma \, z)) \leq c  \det(y_1)^{k/2}   e^{-2 \pi \tr(y_1^{-1}S[v_1])},
\] 
where we have written $\gamma (\tau,w) = (\tau_1,w_1)$. By \cite[Proposition 2]{St} there is a constant $c_0$ such that
\[
\phi(z) \leq  c_0 \left( \prod_{j=1}^n (\lambda_j^{k/2} + \lambda_j^{-k/2} )\right) e^{-2 \pi \tr(y_1^{-1}S[v_1])},
\] 
where $\lambda_j$ are the eigenvalues of $y$. From the description of $\b{\Omega}$, we may take $v_1 = \lambda y_1$ where $\lambda \in M_{l,n}(\mathbb{R})$ with $|\lambda_{ij}| \leq 1$, that is,
\[
e^{-2 \pi \tr(y_1^{-1}S[v_1])} = e^{-2 \pi \tr(S[\lambda] y_1)} .
\]  
But this last expression is bounded for $y_1 \in \b{\Omega}$ since $y_1 > \epsilon 1_n$ for some $\epsilon > 0$ and $\lambda_{ij}$ take values in a compact set. Hence, in fact, there is a constant $c_1 \in \mathbb{R}_+$ such that
\[
\phi(z) \leq  c_1 \left( \prod_{j=1}^n (\lambda_j^{k/2} + \lambda_j^{-k/2} )\right),
\]
and hence we obtain our claim by using $\phi(z) = \det(y)^{k/2} e^{-2 \pi \tr(y^{-1}S[v])} |g(z)|$. 
\end{proof}

\begin{lem} \label{Bounded growth} Let $h\in \mathfrak{M}^n_{k,S}(\b{\Gamma}' ,\psi)$, for some congruence subgroup $\b{\Gamma}'$ and character $\psi$. Assume $h(z) = g(\tau) \phi(\tau,w)$ for some nearly holomorphic Siegel modular form $g(\tau)$ of weight $k-l/2$ and a holomorphic Jacobi form $\phi(\tau,w)$ of weight $l/2$. Then $h\in \mathcal{B}^n_{k,S}(\b{\Gamma}')$ if 
\[
\int_{Y_n}\int_{X_n}|g(\tau) | \left(\prod_{i=1}^n (1- \lambda_j^{-l/2}) \right) e^{-\pi  ( \epsilon \tr(y)) } (\det y)^{(k-l/2)-(n+1)} dxdy < \infty
\]
for all $\epsilon > 0$, where $\lambda_j = \lambda_j(y)$ are the eigenvalues of the symmetric matrix $y>0$.
\end{lem}

\begin{proof}

By Lemma \ref{Bound Jacobi form}, a sufficient condition on the bounded growth of $h(z)$ can be given by finiteness of the following integral: 
\begin{align*}
&\int_{Y_n}\int_{M_{l,n}(\R)}\int_{X_n}\int_{U_n} |g(\tau) | \left(\prod_{i=1}^n (1- \lambda_j^{-l/2}) \right)   e^{2 \pi \tr(S[v]y^{-1})} e^{-2 \pi\tr\(\frac{t_0}{\lagp}y+\T{r_0}v\)}\Delta_{S,k}(z)dz\\
&=vol(U_n)\int_{Y_n}\int_{X_n}|g(\tau) | \left(\prod_{i=1}^n (1- \lambda_j^{-l/2}) \right) e^{-2 \pi\tr\(\frac{t_0}{\lagp}y\)}(\det y)^{k}\\
&\hspace{8cm}\cdot\int_{M_{l,n}(\R)} e^{-2 \pi \tr(S[v]y^{-1}+\T{r_0}v)}dv d\tau\\
&\stackrel{\eqref{eq:cool-id}}{=}vol(U_n)\det(2S)^{-n/2}\\ 
&\hspace{1cm}\int_{Y_n}\int_{X_n}|g(\tau) | \left(\prod_{i=1}^n (1- \lambda_j^{-l/2}) \right) e^{- \frac{\pi}{2\lagp} (\tr (4t_0 - \lagp S^{-1}[r_0])y) } (\det y)^{(k-l/2)-(n+1)} dxdy.
\end{align*}
Since $4t_0 -\lagp S^{-1}[r_0] > 0 $, the above expression is finite if 
\[
\int_{Y_n}\int_{X_n}|g(\tau) | \left(\prod_{i=1}^n (1- \lambda_j^{-l/2}) \right) e^{-\pi  ( \epsilon \tr(y)) } (\det y)^{(k-l/2)-(n+1)} dxdy < \infty
\]
for all $\epsilon > 0$, where we recall that $\lambda_j = \lambda_j(y)$ are the eigenvalues of the symmetric matrix $y>0$. This is similar to the condition of Sturm for a weight $k-l/2$ nearly holomorphic form multiplied by a holomorphic modular form of weight $l/2$ (compare with the proof of \cite[Corollary 2]{St}). There is however a difference in the power of the determinant. In \cite{St} we have $k-(n+1)$ and in our case we have $k-l/2 -(n+1)$.  \end{proof}

\begin{prop} \label{Eisenstein growth} With notation as in section 3,  
\[
E^{2n}|_{k,S}\b{\rho}(z,s,\chi) \in \mathcal{B}^{2n}_{k,S}
\]
for $s \in \frac{1}{2}\mathbb{Z}$ such that $\frac{k-l}{2} - 2n > s > \frac{2n+1 + l/2}{2}$.
\end{prop}
\begin{proof} 
We write $\b{\Gamma}'$ for the congruence subgroup of the Eisenstein series $E^{2n}(z,s,\chi)$. As it was shown in \cite[Section 8]{BMana}, generalizing a previous result of Heim in \cite[Theorem 3.6]{Heim}, the Eisenstein series $E^{2n}(z,s,\chi)$ can be expressed as the trace from a congruence subgroup $\b{\Gamma_1} \subset \b{\Gamma}'$ of some finite index $m$ of a form $E_{k-l/2}^{2n}(\tau,s-l/4) \theta (\tau,w)$, where $E_{k-l/2}^{2n}(\tau,s-l/4)$ is a Siegel type Eisenstein series of weight $k-l/2$, and $\theta(\tau,w)$ a theta series in $M_{l/2,S}$; that is, 
\[
E^{2n}(z,s,\chi) = \sum_{i=1}^m E_{k-l/2}^{2n}(\tau,s-l/4) \theta (\tau,w) |_{k,S} \b{\gamma_i} ,
\]
where the sum is over a set of representatives $\b{\gamma}_i \in \b{\Gamma}' / \b{\Gamma_1}$. From this we obtain
\[
E^{2n}|_{k,S}\b{\rho}(z,s,\chi) = \sum_{i=1}^m E_{k-l/2}^{2n}(\tau,s-l/4) \theta (\tau,w) |_{k,S} \b{\gamma_i} \b{\rho}.
\]

If we write $\gamma_i$ and $\rho$ for the Siegel part of $\b{\gamma}_i$ and $\b{\rho}$ respectively, we see that the growth of the series $E^{2n}_{k-l/2}(\tau,s - l/4)|_k \gamma_i \rho$, for each $\gamma_i$, is bounded by the growth of the series
\[
H_{k-l/2}(z,s) = \det(y)^{s-k/2} \sum  \det(cz + d)^{-2s+l/2},
\] 
where the sum is over all coprime matrices $c,d \in M_{2n}(\mathbb{Z})$. This growth of this series was studied by Sturm in \cite{St}. Strictly speaking Sturm considers integral weight series, namely the case of $k-l/2 \in \mathbb{Z}$, however since we are actually interested in the absolute value $|H_{k-l/2}(z,s)|$  his result is valid also for the case of $k-l/2 \in \frac{1}{2}\mathbb{Z}$, by using the fact, see \cite[Theorem A.2.4 (2)]{Sh00}, that the absolute value of the half-integral factor of automorphy is equal to the square root of the absolute value of integral one, i.e $|h(\gamma,z)| = |j(\gamma,z)|^{1/2}$, in the notation of \cite{Sh00}. For more on how the results of Sturm can be extended to the case of half-integral Siegel modular forms we refer to \cite{Mercuri}. Furthermore we note that the statements in \cite{St} are valid for any congruence subgroup and hence in particular valid for the groups $\rho^{-1} \gamma^{-1}_i \Gamma' \gamma_i \rho$. 

We can now use Lemma \ref{Bounded growth} above to establish the range where the Eisenstein series is of bounded growth. For this we need to modify slightly the range obtained by Sturm in \cite[Corollary 2]{St} to accommodate for the small difference in the exponent of the determinant observed in the condition of bounded growth in the proof of Lemma \ref{Bounded growth}. Following the proof of \cite[Corrolary 2]{St} we see that the bounds are
\[
-2n - \frac{l}{2} > s-\frac{k}{2} > \frac{(2n+1) - (k-l/2)}{2},
\]
or equivalently
\[
\frac{k-l}{2} -2n > s > \frac{2n+1 + l/2}{2}.
\]
\end{proof}

\begin{rem} \label{extension to totally real} In order to extend the results of this section to the case of totally real field, one needs to extend \cite[Corollary 2]{St} to the case of totally real fields. We have not attempted doing this.
\end{rem}

\begin{lem}\label{lem:G} Let 
$$G(z) := \pi^{-\beta}\Lambda_{k-l/2,\mathfrak{c}}^{2n}(\mu/2-l/4,\chi \psi_S) (E^{2n}|_{k,S} \b{\rho}) (z,\mu/2; \chi)$$ 
and assume that  $\frac{k-l}{2} -2n > \frac{\mu}{2} > \frac{2n+1 + l/2}{2}$, $k - \mu \in 2 \mathbb{Z}$, and $\beta$ is as in Theorem \ref{algebraic properties of Eisenstein series}. Then $G \in \mathcal{B}^{2n}_{k,S} \cap N^{2n,D}_{k,S}(\overline{\mathbb{Q}})$, where $D = 2n\left(\frac{k-\mu}{2}\right)$. Moreover, the Fourier coefficients of $G$ are of the form $det(y)^{-D/2n} P(y)$, where $P \in \overline{\mathbb{Q}}[y_{ij}]$ is a polynomial of degree at most $D$.
\end{lem}

\begin{proof} The first statement of the Lemma 
is already contained in Proposition \ref{Eisenstein growth} and Theorem \ref{algebraic properties of Eisenstein series}. 
The new information is the nature of the Fourier coefficients of $G$, which a priory are the polynomials (with algebraic coefficients and of degree at most $D$) in the entries of $y^{-1}$. We have seen in the above proposition that up to a constant the Eisenstein series $G(z)$ is of the form 
\[
G(z) =  \sum_{i=1}^m \left(E_{k-l/2}^{2n}(\tau,\mu/2-l/4) \theta (\tau,w)\right) |_{k,S} \b{\gamma_i} \b{\rho}.
\] 
In particular, it is enough to establish the statement for the nearly holomorphic Siegel modular forms $E_{k-l/2}^{2n}(\tau,\mu/2-l/4)|_{k-l/2} \gamma_i \rho$, where $\gamma_i$ and $\rho$ denote the Siegel part of $\b{\gamma}_i$ and $\b{\rho}$, respectively. It is well known that the Eisenstein series $E_{k-l/2}^{2n}(\tau,\mu/2-l/4)$ are of the form $\delta^{\frac{k-\mu}{2}}_{\mu - l/2} E^{2n}_{\mu-l/2}(\tau)$, where $E^{2n}_{\mu-l/2}(\tau)$ are holomorphic and $\delta^{\frac{k-\mu}{2}}_{\mu - l/2}$ are the Shimura-Maass differential operators (see for example \cite[page 135]{CP} for $l \in 2 \mathbb{Z}$ and \cite[pages 145-146]{Sh00} for any $l \in \mathbb{Z}$). Moreover, if a nearly holomorphic modular form is obtained by such operators from a holomorphic modular form, then the coefficients of its Fourier expansion are of the form $\det(y)^{-(k-\mu)/2} P(y)$ where $P$ is a polynomial in the entries of $y$ of degree at most $2n(k-\mu)/2$ (see \cite[page 106]{CP}). Moreover, the operators $\delta^r_m$ have the property (see \cite[page 116]{CP} for $m \in \mathbb{Z}$ and \cite[Paragraph 14.14]{Sh00} for $m \in\Z+ \frac{1}{2}$) that $\delta^r_m(f |_{m} \gamma) = (\delta_m^r f)|_{m+2r} \gamma$ for a Siegel modular form $f$ of weight $m$ and any $\gamma \in \Sp_{2n}(\mathbb{R})$. In particular, 
\begin{align*}
E_{k-l/2}^{2n}(\tau,\mu/2-l/4)|_{k-l/2} \gamma_i \rho &= \left(\delta^{\frac{k-\mu}{2}}_{\mu - l/2} E^{2n}_{\mu-l/2}(\tau) \right)|_{k - l/2}\gamma_i \rho\\
&= \delta^{\frac{k-\mu}{2}}_{\mu - l/2} \left(E^{2n}_{\mu-l/2}(\tau) |_{\mu - l/2}\gamma_i \rho \right), 
\end{align*}
which proves that the Fourier coefficients of $E_{k-l/2}^{2n}(\tau,\mu/2-l/4)|_{k-l/2} \gamma_i \rho$ are of the required form. 
\end{proof}

The next Lemma shows that the pullback $\Delta^*$ defined in Lemma \ref{diagonal restriction preserves algebraicity} preserves the bounded growth property.

\begin{lem} \label{diagonal bounded} Let $f \in \mathcal{B}^{2n}_{k,S}$ and let $g(z_1,z_2) := \Delta^*(f)$. Then $g$ is of bounded growth in both variables $z_1$ and $z_2$. If moreover $f \in N^{2n,r}_{k,S}$ for some $r \in \mathbb{N}$, then we may write $g(z_1,z_2) = \sum_{i=1}^m g_i(z_1) h_i(z_2)$, where $g_i,h_i \in \mathcal{B}^n_{k,S} \cap N^{n,r}_{k,S}$.
\end{lem}
\begin{proof} We first prove that the growth condition is satisfied for $g(z_1,z_2)$ both for $z_1$ and $z_2$. Since the argument is symmetric we may keep $z_2$ fixed and consider the function $g(z_1) := g(z_1,z_2)$. Let $\b{\Gamma}'$ be the congruence subgroup for $g(z_1)$ and $\lagp$ the  associated constant. We need to show that for any $(t,r)$ with $t \in \frac12 Sym_n(\Z)$ half-integral and $r \in M_{l,n}(\Z)$ such that $4t - \lagp S^{-1}[r] >0$ we have
\[
\int_{Y_n}\int_{M_{l,n}(\R)}\int_{X_n}\int_{U_n} |g(z_1)|e^{-2\pi \tr(t y_1/\lagp)}e^{-2\pi\tr (\T{r}v_1)}\Delta_{S,k}(z_1)d_nz_1 < \infty.
\] 
Since $f$ if of bounded growth, we already know that for any given $(t_0,r_0)$, now of degree $2n$, such that $4t_0 - \lagp S^{-1}[r_0] > 0$ we have 
$$\int_{Y_{2n}}\int_{M_{l,2n}(\R)}\int_{X_{2n}}\int_{U_{2n}} |f(z)|e^{-2\pi \tr(t_0y/\lagp)}e^{-2\pi\tr (\T{r_0}v)}\Delta_{S,k}(z)d_{2n}z < \infty .$$
We note here that since $g$ is simply the pull back of $f$ with respect to the diagonal map $\Delta$ we can use the same $\lagp$.

We specialize to $t_0 := \diag[t ,\,  1_n]$ and $r_0 = [r \,\,\, 0_{l,n}]$ with $t, r$ as above, so that $4t_0 - \lagp S^{-1}[r_0] > 0$. Moreover, in the second integral we decompose
$X_{2n} = \Delta(X_n) \cup \Delta(X_n)^c$, where $\Delta(X_n) = \diag[X_n,\,X_n]$ and $\Delta(X_n)^c$ is its complement in $X_{2n}$; and do the same with $Y_{2n}$. With these choices the above integral is equal to
\begin{equation}\tag{$\star$}\label{eq:lembg}
\begin{split}
\int_{\Delta(Y_{2n)}}\int_{M_{l,2n}(\R)}\int_{\Delta(X_{n})}\int_{U_{2n}} |g(z_1,z_2)|e^{-2\pi \tr(t y_1/\lagp + y_2/\lagp)}e^{-2\pi\tr (\T{r}v_1)}\Delta_{S,k}(z)d_{2n}z\, +\\
\int_{\Delta(Y_{2n})^c}\int_{M_{l,2n}(\R)}\int_{\Delta(X_{n})^c}\int_{U_{2n}} |f(z)|e^{-2\pi \tr(t_0y/\lagp)}e^{-2\pi\tr (\T{r_0}v)}\Delta_{S,k}(z)d_{2n}z < \infty . 
\end{split}\end{equation}

Now consider the function 
$$h(z_2) := \int_{Y_{n}}\int_{M_{l,n}(\R)}\int_{X_{n}}\int_{U_{n}} |g(z_1,z_2)|e^{-2\pi \tr(t y_1/\lagp)}e^{-2\pi\tr (\T{r}v_1)}\Delta_{S,k}(z_1)d_{n}z_1 .$$ 
It is continuous in $z_2$, since $g$ is continuous in the variable $z_2$.  

Suppose that for the selected $(t,r)$ there exists a $z'_2$ such that $g(z_1,z'_2)$ is not of bounded growth in $z_1$, that is, 
\[
h(z'_2) = \int_{Y_n}\int_{M_{l,n}(\R)}\int_{X_n}\int_{U_n} |g(z_1,z'_2)|e^{-2\pi \tr(t y_1/\lagp)}e^{-2\pi\tr (\T{r}v_1)}\Delta_{S,k}(z_1)d_nz_1 = \infty.
\]
Since $h$ is continuous, for every $N > 0$ there exists a neighborhood of $z'_2$, say $D(z_2',\delta)$ such that $|h(z_2)| > N$ for all $z_2 \in D(z_2',\delta)$. This means that the function $h$ is unbounded in any neighborhood of $z_2'$. We will show that this further implies that the first integral in \eqref{eq:lembg}, call it $I$, is infinite.

Choose a compact neighborhood $D$ of $z'_2$ and denote by $M$ be the smallest value of the continuous function $e^{-2\pi \tr( y_2)/\lagp}$ on $D$. Then 
\begin{align*}
&I \geq 
\int_{\Delta(Y_{2n)}}\int_{M_{l,2n}(\R)}\int_{\Delta(X_{n})}\int_{U_{2n}} |g(z_1,z_2)|  \mathbf{1}_D(z_2) e^{-2\pi \tr(t y_1/\lagp + y_2/\lagp)}\\
&\hspace{10.3cm}\cdot e^{-2\pi\tr (\T{r}v_1)}\Delta_{S,k}(z)d_{2n}z \\
&\geq M \int_{\Delta(Y_{2n)}}\int_{M_{l,2n}(\R)}\int_{\Delta(X_{n})}\int_{U_{2n}} |g(z_1,z_2)|  \mathbf{1}_D(z_2) e^{-2\pi \tr(t y_1/\lagp)}e^{-2\pi\tr (\T{r}v_1)}\Delta_{S,k}(z)d_{2n}z\\
&=M \int_{Y_n}\int_{M_{l,n}(\R)}\int_{X_n}\int_{U_n} h(z_2) \mathbf{1}_D(z_2) \Delta_{S,k}(z_2)d_{n}z_2,
\end{align*}
where $\mathbf{1}_D(z_2)$ is the indicator function of the set $D$. Now let $K$ be the smallest value of the continuous function $\Delta_{S,k}(z_2)\det(y_2)^{-(l+n+1)}$ on $D$. Then the last integral is larger or equal to
\[
M K \int_D h(z_2) dx_2dy_2dv_2 du_2 .
\]

Since $h$ is unbounded on $D$, the last integral is undefined and hence so is $I$ (note that all the integrands are positive valued), which is in contradiction with \eqref{eq:lembg}. Since $(t,r)$ were selected arbitrarily, we conclude that this holds for any pair $(t,r)$ and hence $g(z_1,z_2)$ is of bounded growth in $z_1$ for all $z_2$. A symmetric argument establishes also the bounded growth in the $z_2$ variable. 

For the last statement of the lemma we use an argument similar to the proof \cite[Lemma 24.11]{Sh00}. We first observe that if $f$ is additionally nearly holomorphic, then $g$ is both of bounded growth and nearly holomorphic in $z_1$ and $z_2$. Since the space of nearly holomorphic Siegel-Jacobi forms (of a given level) is of finite dimension (Corollary \ref{cor:N_finite_dim}), so is the linear subspace of the nearly holomorphic functions of bounded growth. If $\{ g_i:i=1,\ldots ,m\}$ forms a basis of such functions, then for any fixed $z_2$ we may write $g(z_1,z_2) = \sum_{i=1}^m g_i(z_1) h_i(z_2)$ for some complex numbers $h_i(z_2)$. We can find enough values of $z_1$, say $a_1,\ldots,a_m$ such that the linear system $g(a_j, z_2) = \sum_{i=1}^m g_i(a_j) h_i(z_2)$, $j=1,\ldots, m$, with the unknowns $h_i(z_2)$ is non-singular, and hence solving for the $h_i(z_2)$ we see that they have the same properties (of being bounded and nearly holomorphic) as the $g(a_j,z_2)$.  
\end{proof}

We now let $\omega$ be a CM point of $\mathbb{H}_{n}$ in the sense of \cite[Section 9]{Sh00}, $ v \in M_{l,2n}(\mathbb{Q})$ and $\mathfrak{P}:= \mathfrak{P}_k(\omega) \in \mathbb{C}^{\times}$ be the corresponding period of $\omega$ of weight $k$, defined for example in \cite[Paragraph 11.12]{Sh00}. We set $\b{\omega} := (\omega, v \Omega_{\omega})$ where $\Omega_{\omega} := \transpose{(\omega \,\,\,1_n)}$ and consider $G_*(\b{\omega},z_2):= G_*( (\omega, v \Omega_{\omega}),z_2)$, where $G_*(z_1,z_2) = e(\tr(S[w_1] (\tau_1 - \bar{\tau}_1)^{-1})) G(z_1,z_2)$, $G(z_1,z_2) := \Delta^*(G)$ and $G$ is as in Lemma \ref{lem:G}; this is in agreement with the notation of Proposition \ref{Hecke Operators preserve field of definition}. 

\begin{lem}  \label{Special Fourier Coefficients} 
The Fourier coefficients of the function
$$g_{\omega,v}(z_2) := \frac{1}{\mathfrak{P}} G_*(\b{\omega},z_2) \in \mathcal{B}^n_{k,S} \cap N_{k,S}^{n,D}(\overline{\mathbb{Q}})$$
are of the form $\det(y_2)^{-D/2n}P(y_2)$, where  $D = 2n\left(\frac{k-\mu}{2}\right)$ and $P(y_2) \in \overline{\mathbb{Q}}[y_{2,ij}]$.
\end{lem}
\begin{proof} By Lemmas \ref{lem:G} and \ref{diagonal bounded}, $G(z_1,z_2) = \sum_i \b{f}_i(z_1) \b{g}_i(z_2)$ for some $\b{f}_i, \b{g}_i \in \mathcal{B}^n_{k,S} \cap N^{n,D}_{k,S}(\overline{\mathbb{Q}})$. Hence,
\[
g_{\omega,v}(z_2) = \frac{1}{\mathfrak{P}}  G_*(\b{\omega},z_2) = \frac{1}{\mathfrak{P}}  \sum_i \b{f}_{i*}(\b{\omega}) \b{g}_i(z_2)\in \mathcal{B}^n_{k,S}\, .
\]

As we have seen in section \ref{sec:nhol}, $\b{f}_{i*} \in N_{k}^{n,D}(\overline{\mathbb{Q}})$, and thus by \cite[Paragraph 14.4]{Sh00}, $\mathfrak{P}^{-1}   \b{f}_{i*}(\b{\omega}) \in \overline{\mathbb{Q}}$, which in turn implies that $g_{\omega,v} \in N^{n,D}_{k,S}(\overline{\mathbb{Q}})$.

On the other hand, $G(z_1,z_2)$ is the diagonal restriction of a series with Fourier expansion of the form $\det(y)^{-D/2n} \sum_{t,r}Q_{t,r}(y) e(\tr\(\frac{t}{\lagp} \tau + \T{r}w\))$, where $Q_{t,r} \in \overline{\mathbb{Q}}[y_{ij}]$ are polynomials of degree at most $D$ and $\lagp$ is the constant determined by the level of $G$. Hence,
\begin{align*}
g_{\omega,v}(z_2) =&\frac{1}{\mathfrak{P}}e(\tr(S[v\Omega_{\omega}] (\omega - \bar{\omega})^{-1})) \det(\Im(\omega))^{-D/2n} \det(y_2)^{-D/2n} \\
&\cdot\sum_{t_2,r_2} A_{t_2,r_2}(y_2)e(\tr\(\frac{t_2}{\lagp}\tau_2 + \T{r_2}w_2\))
\end{align*}
where $A_{t_2,r_2}(y_2) = \sum_{t_1,r_1}Q_{\diag[t_1, t_2], [r_1\, r_2]}(\diag[\Im(\omega), y_2])  e(\tr\(\frac{t_1}{\lagp}\omega + \T{r_1} v \Omega_\omega\))$. Observe that $A_{t_2,r_2}$ is a polynomial in $\mathbb{C}[y_{2,ij}]$: it can be written as a sum of infinitely many polynomials, but all of them are of degree at most $D$. However, since $g_{\omega,v}\in N^{n,D}_{k,S}(\overline{\mathbb{Q}})$, it follows that in fact $\mathfrak{P}^{-1} e(\tr(S[v\Omega_{\omega}] (\omega - \bar{\omega})^{-1}))\det(\Im(\omega))^{-D/2n} A_{t_2,r_2}(y_2) \in \overline{\mathbb{Q}}[y_{2,ij}]$. In particular, the Fourier coefficients of $g_{\omega,v}(z_2)$ are of the form $\det(y_2)^{-D/2n}P(y_2)$, where $P\in \overline{\mathbb{Q}}[y_{2,ij}]$.
\end{proof}

We can now establish the analogue of Theorem \ref{ratio of inner products} without assuming Property A. Its proof is different from the one of Theorem \ref{ratio of inner products}, since instead of using Fourier expansion to characterize Siegel-Jacobi forms, we use their values at CM points. This allows us to take the lower bound for $k$ as small as possible.
The main issue with the first approach is that if one writes, with the notation above, $G(z_1,z_2) = \sum_i \b{f}_i(z_1) \b{g}_i(z_2)$ then we do not know whether the $\b{f}_i,\b{g}_i \in N^{n,D}_{k,S}$ have Fourier coefficients of the form $\det(y)^{-D/2n} P(y)$, which would allow us to take a smaller bound for $k$ when we apply the holomorphic projection (cf. Corollaries \ref{Hol preserves algebraicity} and \ref{improved bound projection}). We note here that a similar idea (evaluation at CM points) was used in \cite[section 29]{Sh00} on results of non-splitting unitary groups to compensate for the lack of Fourier expansion in the classical sense.

\begin{thm} \label{ratio of inner products_v2}
Assume $n >1$, and let $0 \neq \b{f} \in S_{k,S}^{n}(\b{\Gamma},\overline{\mathbb{Q}})$ be an eigenfunction of $T(\mathfrak{a})$ for all integral ideals $\mathfrak{a}$ with $(\f{a},\f{c})=1$. Assume that $k > 6n+2l + 1$ and that the matrix $S$ satisfies the condition $M_{\f{p}}^+$ for every prime ideal $\f{p}$ of $\Q$ coprime to $\f{c}$.
Then for any $\b{g} \in S_{k,S}^n(\overline{\mathbb{Q}})$,
\[
\frac{<\b{f},\b{g}>}{ <\b{f}, \b{f}>} \in \overline{\mathbb{Q}}.
\]
\end{thm}

\begin{proof} As we argued in the proof of Theorem \ref{ratio of inner products}, thanks to the decomposition
\[
S_{k,S}^{n}(\b{\Gamma},\overline{\Q}) = V(\b{f}) \oplus \b{U},
\]
where $\b{U}$ is a $\overline{\mathbb{Q}}$-vector space orthogonal to $V(\b{f})$, we may assume, without loss of generality, that $\b{g} \in V(\b{f})$. We now select $\mu \in k+2 \mathbb{Z}$ such that
\[
\frac{k-l}{2} -2n > \frac{\mu}{2} > \frac{2n+1 + l/2}{2}.
\] 

As in the proof of Theorem \ref{ratio of inner products} we consider a character $\chi$ of conductor $\f{f}_{\chi} \neq \mathfrak{o}$ such that $\chi_{\a}(x) = \sgn_{\a}(x)^{k}$, $\chi^2 \neq 1$ and $G(\chi,\mu-n-l/2) \in \overline{\mathbb{Q}}^{\times}$, where 
$G(\chi,\mu-n-l/2)$ is as in Theorem \ref{thm:dm_identity}, equation \eqref{definition of $G$}. Then for any $\tilde{\b{f}} \in V(\b{f})$ we can evaluate the doubling identity of Theorem \ref{thm:dm_identity} at $s=\mu/2$ to obtain
\begin{multline*}
\Lambda_{k-l/2,\mathfrak{c}}^{2n}(\mu/2-l/4,\chi \psi_S)  vol(A)  <(E|_{k,S} \b{\rho}) (\diag[z_1,z_2],\mu/2;\chi),(\tilde{\b{f}}^c|_{k,S}\b{\eta}_n)^c(z_2)>\\
={^{\overline{\mathbb{Q}}}}^{\times} c_{S,k}(\mu/2-k/2) \b{\Lambda}(\mu/2,\b{f},\chi) \tilde{\b{f}}^c(z_1), 
\end{multline*}
where we recall that $\tilde{\b{f}}^c(z_1)$ is the Jacobi form obtained from $\tilde{\b{f}}(z_1)$ by taking complex conjugation on the Fourier coefficients.

In particular, for
\[
G(z_1,z_2) = \pi^{-\beta}\Lambda_{k-l/2,\mathfrak{c}}^{2n}(\mu/2-l/4,\chi \psi_S) (E|_{k,S} \b{\rho}) (\diag[z_1,z_2],\mu/2; \chi),
\] 
the diagonal restriction of $G(z)$ defined in Lemma \ref{lem:G}, that is, with $\beta$ as in Theorem \ref{algebraic properties of Eisenstein series}, we obtain
\[
 <G(z_1,z_2), \b{F}(z_2) > ={^{\overline{\mathbb{Q}}}}^{\times} \pi^{\delta- d_0 +\beta} \b{\Lambda}(\mu/2,\b{f},\chi) \tilde{\b{f}}^c(z_1),
\]
where $\b{F}:= (\tilde{\b{f}}^c|_{k,S}\b{\eta}_n)^c=\tilde{\b{f}}|_{k,S}\b{\eta}^{-1}_n\in S_{k,S}^n(\overline{\mathbb{Q}})$. Then, for a CM point $\omega \in \mathbb{H}_n$ and $v \in M_{l,2n}(\mathbb{Q})$, with the notation as in Lemma \ref{Special Fourier Coefficients},
\[
 <G_*(\b{\omega},z_2), \b{F}(z_2) > ={^{\overline{\mathbb{Q}}}}^{\times} \pi^{\delta- d_0 +\beta} \b{\Lambda}(\mu/2,\b{f},\chi) \tilde{\b{f}}_*^c(\b{\omega}),
\]
or equivalently
\[
<g_{\omega,v}(z_2), \b{F}(z_2) > ={^{\overline{\mathbb{Q}}}}^{\times} \pi^{\delta- d_0 +\beta} \b{\Lambda}(\mu/2,\b{f},\chi) \mathfrak{P}_k(\omega)^{-1}\tilde{\b{f}}_*^c(\b{\omega}),
\]
where $\mathfrak{P}_k(\omega) \in \mathbb{C}^{\times}$ is the CM period corresponding to $\omega$. By the lemma above, $g_{\omega,v}\in N^{n,D}_{k,S}(\overline{\Q})$ with $D= n(k-\mu)$. We note that $k> n + l/2 + (k-\mu)$  since we have selected the $\mu$ such that $\mu > 2n+ l/2 +1$. In particular, we can employ Lemma \ref{Special Fourier Coefficients} and Corollary \ref{improved bound projection}, and consider  $\b{h}_{\omega,v}(z_2) := Hol(g_{\omega,v}(z_2)) \in S^n_{k,S}(\overline{\mathbb{Q}})$ to obtain
\begin{equation}\tag{$\ast$}\label{eq:alg}
<\b{h}_{\omega,v}(z_2),\b{F}(z_2) > ={^{\overline{\mathbb{Q}}}}^{\times} \pi^{\delta- d_0 +\beta} \b{\Lambda}(\mu/2,\b{f},\chi) \mathfrak{P}_k(\omega)^{-1}\tilde{\b{f}}_*^c(\omega).
\end{equation} 

Since $\tilde{\b{f}} \in V(\b{f})$ was arbitrary, the forms $\widetilde{\b{h}}_{\omega,v} := \b{h}_{\omega,v}|_{k,S}\b{\eta}_n \in S_{k,S}^n(\overline{\mathbb{Q}})$  (or rather their projections to $V(\b{f})$) for the various $(\omega,v)$ span the space $V(\b{f})$ over $\overline{\mathbb{Q}}$. Indeed, if we denote by $\mathcal{S} \subset V(\b{f})$ the $\overline{\mathbb{Q}}$ vector space spanned by the projections of $\widetilde{\b{h}}_{\omega,v}$ to $V(\b{f})$ and if there exists an $\tilde{\b{f}} \in V(\b{f})$ which is not in $\mathcal{S}$ then there is a form, say $\tilde{\b{f}_1} \in V(\b{f})$ which is orthogonal to $\mathcal{S}$. But then this would imply by \eqref{eq:alg} above that 
$\tilde{\b{f}}_{1*}^c(\omega) = 0$ for all $(\omega,v)$, since $\b{\Lambda}(\mu/2,\b{f},\chi) \neq 0$ for $\mu > 2n + l +1$ by the statement \eqref{eq:non-vanishing of L-values}. In particular, since the CM points are dense in $\mathbb{H}_n$ (see \cite[page 77]{Sh00}), we conclude that $\tilde{\b{f}}_{1*}=0$ and thus so is $\tilde{\b{f}_1} = 0$. Hence, indeed, $\widetilde{\b{h}}_{\omega,v}$ span $V(\b{f})$.
Moreover, for any $(\omega,v)$,
\[
<\widetilde{\b{h}}_{\omega,v},\tilde{\b{f}}> \in \pi^{\delta- d_0 +\beta} \b{\Lambda}(\mu/2,\b{f},\chi) \overline{\mathbb{Q}}.
\]  
This means that for any $\b{g} \in V(\b{f})$, and taking $\tilde{\b{f}}$ above to be equal to $\b{f}$, we obtain $<\b{g},\b{f}> \in \pi^{\delta- d_0 +\beta} \b{\Lambda}(\mu/2,\b{f},\chi) \overline{\mathbb{Q}}$. In particular, the same holds for $\b{g} = \b{f}$, and that concludes the proof by observing that of course $<\b{f},\b{f}> \neq 0$.
\end{proof}

We are now ready to give the proof of Theorem \ref{Main Theorem on algebraicity v2}.

\begin{proof}[Proof of Theorem \ref{Main Theorem on algebraicity v2}] We follow the same steps as in the proof of Theorem \ref{ratio of inner products_v2}, but now we take $\tilde{\b{f}}$ to be our $\b{f}$ and evaluate the doubling identity of Theorem \ref{thm:dm_identity} at $s = \sigma/2$. . Note also that the restrictions on $\sigma$ in conditions $(i)$ and $(ii)$ are the ones that make the corresponding Eisenstein series nearly holomorphic and of bounded growth. For every CM point $\omega$ of $\mathbb{H}_{n}$ and $v \in M_{l,2n}(\mathbb{Q})$ such that $\b{f}_*^c(\b{\omega})=\b{f}_*^c((\omega, v \Omega_{\omega})) \neq 0$ we obtain as in \eqref{eq:alg} above,
\[
<\b{h}_{\omega,v}(z_2), \b{f}(z_2)>={^{\overline{\mathbb{Q}}}}^{\times} \pi^{\delta- d_0 +\beta} \b{\Lambda}(\sigma/2,\b{f},\chi)
\]
for some $\b{h}_{\omega,v} \in \b{S}_{k,S}^n(\overline{\mathbb{Q}})$, where we have used the fact that $\mathfrak{P}_k(\omega)^{-1}\tilde{\b{f}}_*^c(\b{\omega}) \in \overline{\mathbb{Q}}^{\times}$. 
The exact power of $\pi$ can be derived in exactly the same way as in Theorem \ref{Main Theorem on algebraicity}. Hence,
thanks to Theorem \ref{ratio of inner products_v2}, dividing the above equality by $<\b{f},\b{f}>$ finishes the proof.
\end{proof}


\end{document}